\documentclass[a4paper, english, 11pt, twoside, reqno]{amsart}

\usepackage{babel}
\usepackage[T1]{fontenc}
\usepackage{amssymb}
\usepackage{amsmath, amsfonts, mathrsfs, amscd}
\usepackage{amsthm}
\usepackage{xcolor}
\usepackage{eucal}
\usepackage[all]{xy}
\usepackage{calligra}
\usepackage{enumitem}
\usepackage{graphicx}
\usepackage{hyperref}
\usepackage{pifont}
\usepackage{relsize}
\usepackage{mathtools}
\usepackage{upgreek}
\usepackage{dsfont}
\usepackage{multirow}
\usepackage{tabularx}

\newtheorem*{theorem-intro}{Theorem}
\newtheorem*{question-intro}{Question}
\newtheorem*{corollary-intro}{Corollary}
\newtheorem{theorem}{Theorem}[section]
\newtheorem{lemma}[theorem]{Lemma}
\newtheorem{proposition}[theorem]{Proposition}
\newtheorem{corollary}[theorem]{Corollary}

\theoremstyle{definition}

\newtheorem{remark}[theorem]{Remark}

\newcommand{\Z}{\mathbb{Z}}
\newcommand{\Q}{\mathbb{Q}}
\newcommand{\Ow}{\mathcal{O}}

\newcommand{\Co}{\mathbb{C}}

\newcommand{\F}{\mathbb{F}}
\newcommand{\E}{\mathbb{E}}
\newcommand{\Fq}{\mathbb{F}_q}

\newcommand{\Mcal}{{\mathcal M}}

\newcommand{\B}{\mathbb{B}}

\newcommand\Ima{\operatorname{Im}}

\newcommand\car{\operatorname{char}}

\newcommand\Tr{\operatorname{Tr}}
\newcommand\Aut{\operatorname{Aut}}

\newcommand\G{\mathbb G}

\newcommand\T{\mathbb T}
\newcommand\U{\mathbb U}
\newcommand{\Uno}{\mathds{1}}

\newcommand{\Oint}{\mathcal{O}}

\newcommand{\Fr}{\operatorname{Fr}}

\newcommand{\Ind}{\operatorname{Ind}}

\newcommand{\ord}{\operatorname{ord}}

\newcommand{\overbar}[1]{\mkern 1.5mu\overline{\mkern-1.5mu#1\mkern-3.5mu}\mkern 1.5mu}
\newcommand{\menos}{\hspace{-1pt}\setminus \hspace{-1pt}}
\newcommand*\sq{\mathbin{\vcenter{\hbox{\rule{0.6ex}{0.6ex}}}}}

\newcommand{\GL}{\mathbf{GL}}
\newcommand{\PGL}{\mathbf{PGL}}
\newcommand{\PSL}{\mathbf{PSL}}
\newcommand{\SL}{\mathbf{SL}}

\newcommand{\Sp}{\mathbf{Sp}}

\allowdisplaybreaks

\DeclareMathAlphabet{\mathpzc}{OT1}{pzc}{m}{it}

\definecolor{rojo}{RGB}{196,1,9}
\definecolor{azul}{RGB}{29, 79, 132}
\definecolor{naranja}{RGB}{254,101,18}

\begin{document}

\title[Non-existence of integral Hopf orders]{Non-existence of integral Hopf orders for twists of several simple groups of Lie type}

\author[G. Carnovale, J. Cuadra, E. Masut]{Giovanna Carnovale, Juan Cuadra, Elisabetta Masut}

\address{\noindent G. C. \& E. M.:
Dipartimento di Matematica Tullio Levi-Civita,
Universit\`a degli Studi di Padova,
via Trieste 63, 35121 Padova, Italia}
\address{\noindent J. C.: Universidad de Almer\'\i a, Dpto. Matem\'aticas, 04120 Almer\'\i a, Spain\bigskip}

\email{carnoval@math.unipd.it}
\email{jcdiaz@ual.es}
\email{elisabetta.masut@math.unipd.it}


\begin{abstract}
Let $p$ be a prime number and $q=p^m$, with $m \geq 1$ if $p \neq 2,3$ and \mbox{$m>1$} otherwise. Let $\Omega$ be any non-trivial twist for the complex group algebra of $\PSL_2(q)$ arising from a $2$-cocycle on an abelian subgroup of $\PSL_2(q)$. We show that the twisted Hopf algebra $(\Co \PSL_2(q))_{\Omega}$ does not admit a Hopf order over any number ring. The same conclusion is proved for the Suzuki groups, and for $\SL_3(p)$ when the twist stems from an abelian $p$-subgroup. This supplies new families of complex semisimple (and simple) Hopf algebras that do not admit a Hopf order over any number ring. The strategy of the proof is formulated in a general framework that includes the finite simple groups of Lie type. \par \vspace{2pt}

As an application, we combine our results with two theorems of Thompson and Barry and Ward on minimal simple groups to establish that for any finite non-abelian simple group $G$ there is a twist $\Omega$ for $\Co G$, arising from a $2$-cocycle on an abelian subgroup of $G$, such that $(\Co G)_{\Omega}$ does not admit a Hopf order over any number ring. This partially answers in the negative a question posed by Meir and the second author.
\end{abstract}
\maketitle

\section{Introduction}

Before discussing the content of this paper, we put into context the topic treated here and review the antecedents.

\subsection{Hopf orders in group algebras and Galois module theory}
Hopf orders (see Section 2 for a definition) were proposed by Childs in \cite{Ch1} as a way to deal with wildly ramified field extensions in the study of rings of algebraic integers. The role played by the structure of Hopf algebra is explained in \cite{Ch1} and more profusely in the monograph \cite{Ch2}. The introductions of \cite{By1}, \cite{By2}, \cite{ChM}, \cite{Ta1}, and \cite{Ta2} are also illustrative. \par \vspace{2pt}

Let $L/K$ be a Galois extension of number fields with Galois group $G$ and $\Oint_L/\Oint_K$ the corresponding extension of rings of algebraic integers. Recall that $L$ has a normal basis as a $K$-vector space; that is, a basis of the form $\{g(a): g \in G\}$ for some $a \in L$. This can be rephrased as that $L$ is free (necessarily of rank one) as a module over the group algebra $K\hspace{-0.9pt}G$. Galois module theory revolves around the question of whether $\Oint_L$ admits an integral normal basis as an $\Oint_K$-module; see \cite{F}. A keystone in this theme is Noether's theorem, which states that $\Oint_L$ is locally free (or projective) as a module over the group ring $\Oint_K G$ if and only if $L/K$ is at most tamely ramified. \par \vspace{2pt}

For wildly ramified extensions, Leopoldt successfully used the following (algebra) order of $K\hspace{-0.9pt}G$, which is larger than $\Oint_K G$:
$$\mathcal{A}=\{\alpha \in K\hspace{-0.9pt}G : \alpha(s) \in \Oint_L \hspace{1.3mm} \forall s \in \Oint_L\}.$$
After a first result in \cite{Ch1} for abelian extensions, Childs and Moss proved in \cite{ChM} that if $\mathcal{A}$ is a Hopf order of $K\hspace{-0.9pt}G$ over $\Oint_K$, then $\Oint_L$ is locally free of rank one as an $\mathcal{A}$-module. The range of application of this idea in number theory and algebraic geometry is beautifully expounded in the notes \cite{Ta2} by Taylor and Byott.

\subsection{Kaplansky's sixth conjecture and non-existence of integral Hopf orders}
Whereas for cocommutative Hopf algebras there is a well-established theory of Hopf orders, very little is known for non-commutative and non-cocommutative Hopf algebras, and, especially, for semisimple ones. A problem related to Hopf orders that is largely open here is Kaplansky's sixth conjecture. This conjecture states that the dimension of every irreducible representation of a complex semisimple Hopf algebra $H$ divides the dimension of $H$. In plain words, it claims that Frobenius' divisibility theorem holds for complex semisimple Hopf algebras. Larson proved in \cite{L} that if $H$ admits a Hopf order over a number ring, then $H$ satisfies the conjecture, thus leaving open the question of whether a complex semisimple Hopf algebra always admits a Hopf order over a number ring. \par \vspace{2pt}

In connection with this problem, Meir and the second author initiated in \cite{CM1} and \cite{CM2} a study of the existence of Hopf orders for several families of semisimple Hopf algebras that are constructed as Drinfeld twists of group algebras through the following method proposed by Movshev in \cite{Mv}: \par \vspace{2pt}

Let $G$ be a finite non-abelian group and $M$ an abelian subgroup of $G$. Let $K$ be a number field and consider\vspace{0.5pt} the group algebra $K\hspace{-0.9pt}M$. Assume that $K$ is large enough so that $K\hspace{-0.9pt}M$ is isomorphic\vspace{-1pt} to a direct product of $\vert M \vert$ copies of $K$. Take the set $\{e_{\phi}\}_{\phi \in \widehat{M}}$ of orthogonal primitive\vspace{-3.5pt} idempotents giving that decomposition. Here, $\widehat{M}$ denotes the character group of $M$ and $e_{\phi} = \frac{1}{\vert M \vert } \sum_{v \in M} \phi(v^{-1}) v.$ If $\omega: \widehat{M} \times \widehat{M} \rightarrow K^{\times}$ is a normalized 2-cocycle, then
$$\Omega_{M, \omega}:=\sum_{\phi,\psi \in \widehat{M}} \omega(\phi,\psi) e_{\phi} \otimes e_{\psi}$$
is a twist for $K\hspace{-0.9pt}G$ (see Subsection \ref{drt}). The natural coalgebra structure of $K\hspace{-0.9pt}G$ is deformed by $\Omega_{M, \omega}$ as follows:
$$\Delta_{\Omega_{M, \omega}}(g):= \Omega_{M, \omega} (g \otimes g) \Omega_{M, \omega}^{-1} \qquad \forall g \in G.$$
(See Subsection \ref{drt} for the matching formula of the antipode.) This new Hopf algebra structure on the group algebra $K\hspace{-0.9pt}G$ is denoted by $(K\hspace{-0.9pt}G)_{\Omega_{M, \omega}}$. \par \vspace{2pt}

For the families of groups $G$ and twists $\Omega_{M, \omega}$ studied in \cite{CM1} and \cite{CM2} it was proved that $(K\hspace{-0.9pt}G)_{\Omega_{M, \omega}}$ does not admit a Hopf order over $\Oint_K$. As a consequence, the complex semisimple Hopf algebra $(\Co G)_{\Omega_{M,\omega}}$ does not admit a Hopf order over any number ring. (But it satisfies Kaplansky's conjecture.) This uncovered an important arithmetic difference between semisimple Hopf algebras and group algebras, and showed, in particular, that the proof of Frobenius' theorem does not carry over, at least in the expected way, to all semisimple Hopf algebras. \par \vspace{2pt}

The reason for the non-existence of integral Hopf orders for these families of semisimple Hopf algebras is unclear. In \cite[Section 5]{CM2} the authors focused on the fact that they are simple as Hopf algebras (i.e., the only normal Hopf subalgebras are the trivial ones) and asked the following:

\begin{question-intro}
Let $G$ be a finite non-abelian group. Let $\Omega$ be a non-trivial twist for $\Co G$, arising from a $2$-cocycle on an abelian subgroup of $G$ as before. Suppose that $(\Co G)_{\Omega}$ is simple. Can $(\Co G)_{\Omega}$ admit a Hopf order over a number ring?
\end{question-intro}

When $G$ is simple, $(\Co G)_{\Omega}$ is always simple independently of the twist (\cite[Corollary 4.3]{N}).

\subsection{Subject of the paper}
The present paper elaborates on this question. Its ultimate aim is to give the following partial negative answer (Corollary \ref{simple2}):

\begin{theorem-intro}
Let $G$ be a finite non-abelian simple group. Then, there is a twist $\Omega$ for $\Co G$, of the previous kind, such that $(\Co G)_{\Omega}$ does not admit a Hopf order over any number ring.
\end{theorem-intro}

This is achieved using first a result of Barry and Ward (relying on the classification of finite simple groups) which asserts that $G$ contains a minimal simple group (\cite{BW}). A minimal simple group is a non-abelian simple group all of whose proper subgroups are solvable. Secondly, we invoke the classification of such groups established by Thompson in \cite{Th}. They are:
\begin{enumerate}
\item[$\sq$]  $\PSL_2(2^p)$, with $p$ a prime. \vspace{0.75pt}
\item[$\sq$] $\PSL_2(3^p)$, with $p$ an odd prime. \vspace{0.75pt}
\item[$\sq$] $\PSL_2(p)$, with $p>3$ prime such that $5$ divides $p^2+1$. \vspace{0.75pt}
\item[$\sq$] ${}^2\!B_2(2^p)$, with $p$ an odd prime. \vspace{0.75pt}
\item[$\sq$] $\PSL_3(3)$.
\end{enumerate}
Thirdly, we focus on the families of projective special linear groups of sizes $2$ and $3$ and the Suzuki groups. To deal with them we devise in Section \ref{stfrm} a general strategy and a framework of application that covers the finite simple groups of Lie type. Then we prove the following core result (Theorems \ref{main1} and \ref{main2} and Proposition \ref{Suz:classct} abridged): \enlargethispage{\baselineskip}

\begin{theorem-intro}
Let $K$ be a number field and $R\subset K$ a Dedekind domain containing $\Oint_K$. Let \vspace{-3pt}
\begin{enumerate}
\item[$\sq$] $p$ be a prime number and $q=p^m$ with $m \geq 1$. \vspace{1.5pt}
\item[$\sq$] $G$ be a finite group from the list below. \vspace{1.5pt}
\item[$\sq$] $M$ be the abelian subgroup of $G$ that is given together with $G$ in the list. \vspace{1pt}
\item[$\sq$] $\omega: \widehat{M} \times \widehat{M} \rightarrow K^{\times}$ be any normalized non-degenerate cocycle. (This needs $M$ to be of central type.)
\end{enumerate}
If $(K\hspace{-1pt}G)_{\Omega_{M,\omega}}$ admits a Hopf order over $R$, then $\frac{1}{|M|} \in R$. Hence,
$(K\hspace{-1pt}G)_{\Omega_{M,\omega}}$ does not admit a Hopf order over $\Oint_K$. \par \vspace{2pt}

The pairs $(G,M)$ for which the above statement holds are the following:
\begin{enumerate}
\item[(i)] $G$ is $\PSL_2(q)$ and $M$ is any non-trivial abelian subgroup of central type. (Here, $m>1$ if $p=2,3$.) \vspace{1.5pt}
\item[(ii)] $G$ is  $^2\!B_2(q)$ and $M$ is any non-trivial abelian subgroup of central type. (Here, $p=2$ and $m>1$ odd.) \vspace{1.5pt}
\item[(iii)] $G$ is $\SL_3(p)$ and $M$ is any abelian $p$-subgroup of central type.
\end{enumerate}
\end{theorem-intro}

The statement for the complexified group algebra will follow as a consequence (Corollaries \ref{complexif} and \ref{cormain2} abridged):

\begin{corollary-intro}
Let $(G,M,\omega)$ be as in the above theorem. Then, the complex semisimple Hopf algebra $(\Co G)_{\Omega_{M,\omega}}$ does not admit a Hopf order over any number ring.
\end{corollary-intro}

This supplies more families of complex semisimple and simple Hopf algebras for which the previous question has a negative answer. For $\PSL_2(q)$ and $^2\!B_2(q)$ it provides furthermore the first infinite families of simple groups for which the answer is complete. \par \vspace{2pt}

So far we have presented the main results obtained in this paper and explained how they contribute to the preliminary question. We next outline the proof of the preceding theorem (see Section \ref{stfrm}). \par \vspace{2pt}

Set, for short, $H=(K\hspace{-1pt}G)_{\Omega_{M,\omega}}$ and suppose\vspace{0.5pt} that $X$ is a Hopf order of $H$ over $R$. The dual Hopf order $X^{\star}$ consists of those $\varphi \in H^*$ such that $\varphi(X) \subseteq R$. The whole proof is rooted in the fact that every character of $H$ lies in $X^{\star}$ and every cocharacter of $H$ lies\vspace{1.5pt} in $X$ (Proposition \ref{character}). As in \cite{CM1} and \cite{CM2}, the idea is to get elements $\varphi \in X^{\star}$ and $x \in X$ such\vspace{-2pt} that $\varphi(x)=\frac{n}{\vert M \vert}$, where $n \in \Z$ and $\gcd(n,\vert M \vert)=1$. This will give that $\frac{1}{\vert M \vert} \in R$. Such elements are constructed\vspace{0.5pt} by using a character and a cocharacter of $H$ and the structure of Hopf order of $X$.  \par \vspace{2pt} \enlargethispage{\baselineskip}

For the groups in the theorem, and more generally for any simple group of Lie type, there is a natural Sylow $p$-subgroup $U$ available. The novelty in our approach with respect to \cite{CM1} and \cite{CM2} is the single usage of the induced character $\chi:=\Ind_U^G(\Uno_U)$. By construction, $\chi$ vanishes outside $\bigcup_{g\in G} g \hspace{0.25pt} U\hspace{-1.25pt} g^{-1}$. This character\vspace{0.75pt} equips us with a powerful tool that works in the general framework described in Subsection \ref{sec:strategy-p}. For practical effects, its limitation depends on the difficulty of computing the values at $U$. The computation for the groups treated here is not hard. \par \vspace{2pt}

For the cocharacter, we pick $\tau \in G$ such that $M \cap (\tau M \tau^{-1})=\{1\}$. The subspace spanned by the double coset $M \tau M$ is a simple subcoalgebra of $H$. The associated irreducible cocharacter is $c_{\tau}:=|M|e_{\varepsilon}\tau e_{\varepsilon}$ (Proposition \ref{decomp}). Here, as before, $e_{\varepsilon}$ denotes\vspace{0.75pt} the idempotent of $K\hspace{-0.9pt}M$ attached to the trivial character. The \linebreak element $y_{\chi}:= (\chi \otimes id)\Delta_{\Omega_{M,\omega}}(c_{\tau})$ must belong\vspace{0.5pt} to $X$ and, therefore, $y_{\chi}^2$ must as well. Then, $\chi(y_{\chi}^2) \in R$. These two elements, $\chi \in X^{\star}$ and $y_{\chi}^2 \in X$, will accomplish the task. \par \vspace{2pt}

In Proposition \ref{key} we show that $y_{\chi}$ can be written as:
$$y_{\chi}=\frac{1}{\vert M \vert}\sum_{g \in M \tau M} \chi(g)g.$$
The expression of $\chi(y_{\chi}^2)$ given in Equation \ref{chiychi2fib} allows us to design a procedure to carry out its computation in a systematic manner. The relation between the double coset $M \tau M$ and the fibers of $\chi$ of non-zero values is behind the result. At the end, the tasks reduce to: choose an appropriate representative $\tau$ of the double cosets of $M$ in $G$, compute the non-zero values of $\chi$ and follow the steps in the procedure. For the choice of $\tau$, we propose in Subsection \ref{sec:strategy-p} a meaningful set of elements in any simple group of Lie type. \par \vspace{2pt}

Our results suggest that the answer to the precursory question could be always negative. However, as in \cite{CM1} and \cite{CM2}, the method of proof does not tell us about the reason for this to be so. To make progress on this matter one probably has to understand in depth the effect of the twisting operation inside the classical theory of Hopf orders. A fact that could be related to this is the following, see \cite[Proposition 20.1 and page 108]{Ch2}: {\it Let $L$ be a number field and $G$ a finite group which has no non-trivial normal $p$-subgroups for any prime number $p$. Then, the only Hopf order of $LG$ is $\Oint_L G$.}

\subsection{Organization of the paper} The paper is organized as follows: \par \vspace{2pt}

Section 2 provides the necessary background material: basics on Hopf orders, definition of Drinfeld's twist, and Movshev's method of twisting a group algebra. It also aims to establish Proposition \ref{key}. Section 3 presents the strategy that we will use in the proof and explains how it can be potentially applied to finite groups of Lie type. The procedure to compute $\chi(y_{\chi}^2)$ is described here. \par \vspace{2pt}

In Section 4 we prove the non-existence of integral Hopf orders for the twisted group algebra of $\PSL_2(q)$, $^2\!B_2(q)$, and $\SL_3(p)$ when the twist arises from an abelian $p$-subgroup of central type (Theorem \ref{main1} and Corollary \ref{complexif}). For $^2\!B_2(q)$ these are indeed all the abelian subgroups of central type (Proposition \ref{Suz:classct}). In Section 5 we discuss in full the case of $\PSL_2(q)$. In the first part, we determine all the abelian subgroups of central type. We show in Proposition \ref{subPSL2} that they are: $p$-subgroups of square order or Klein four-groups. In the second part, we prove the non-existence of integral Hopf orders for the twisted group algebra of $\PSL_2(q)$ when the twist arises from the latter ones (Theorem \ref{main2} and Corollary \ref{cormain2}). \par \vspace{2pt} \enlargethispage{\baselineskip}

The application of these results to the twisted group algebra of any finite non-abelian simple group is expounded in Section 6.

\subsection{Acknowledgements}
Giovanna Carnovale and Elisabetta Masut were partially supported by DOR2022390/20 ``\emph{Hopf algebras, Nichols algebras and deformations}'' and BIRD203834 ``\emph{Grassmannians, flag varieties and their generalizations}'', funded by the University of Padova. Giovanna Carnovale participates in the INdAM group GNSAGA and in PRIN 2017YRA3LK-002 ``\emph{Moduli and Lie Theory}''. The work of Juan Cuadra was partially supported by the Spanish Ministry of Science and Innovation, through the grant PID2020-113552GB-I00 (AEI/FEDER, UE), by the Andalusian Ministry of Economy and Knowledge, through the grant P20\underline{ }00770, and by the research group FQM0211. \par \vspace{2pt}

Juan Cuadra would like to thank Ehud Meir for several helpful discussions. Giovanna Carnovale would like to express her gratitude to Jay Taylor for an enlightening e-mail exchange concerning GG-characters. \par \vspace{2pt}

The authors are grateful to the referees for their comments and suggestions, which helped to improve several aspects of the writing of this paper. \vspace{1mm}

\section{Recollections and preliminaries}
\setcounter{equation}{0}

The general background for this paper is the same as that of \cite{CM1} and \cite{CM2}. For convenience, we collect in this section several notions, constructions, and results that will be needed later. We refer to \cite[Subsection 1.2]{CM1} and \cite[Sections 1 and 2]{CM2} for proofs, further explanations, and extra references.

\subsection{Conventions}
We will work over a ground field $K$. Vector spaces, linear maps, and unindexed tensor products are over $K$, unless otherwise specified. Throughout, $H$ is a finite-dimensional Hopf algebra over $K$. We denote by $1_H$ its identity element; and by $\Delta,\varepsilon$, and $S$ its coproduct, counit, and antipode, respectively. The dual Hopf algebra of $H$ is denoted by $H^*$. \par \vspace{2pt}

Our main references for the general theory of Hopf algebras are \cite{Mo} and \cite{Ra}. For the classical theory of Hopf orders (cocommutative Hopf algebras) we recommend \cite{Ch2} and \cite{U}.

\subsection{Recap of Hopf orders} Let $W$ be a finite-dimensional vector space over $K$. Let $R$ be a subring of $K$. A \emph{lattice of\hspace{1pt} $W$\hspace{-1.5pt} over $R$} is a finitely generated and projective $R$-submodule $X$ of $W$ such that the natural map $X \otimes_R K \rightarrow W$ is an isomorphism. The submodule $X$ corresponds to the image of $X \otimes_R R$. \par \vspace{2pt}

A \textit{Hopf order of $H$ over $R$} is a lattice $X$ of $H$ that is closed under all the Hopf algebra operations; namely, $1_H \in X$, $XX \subseteq X$, $\Delta(X)\subseteq X\otimes_{R} X$, $\varepsilon(X) \subseteq R,$ and $S(X)\subseteq X$. (For the coproduct, $X\otimes_{R} X$ is naturally identified as an $R$-submodule of $H\otimes H$.) Alternatively, $X$ is a Hopf algebra over $R$, which is finitely generated and projective as an $R$-module, such that $X\otimes_{R} K \simeq H$ as Hopf algebras over $K$. \par \smallskip

In the following three propositions, $K$ is a number field and $R\subset K$ is a Dedekind domain containing $\Oint_K$, the ring of algebraic integers of $K$. Under these hypotheses, $K$ is the field of fractions of $R$. Hopf order means Hopf order over $R$. \par \smallskip

\begin{proposition}\cite[Lemma 1.1]{CM1} \label{dual}
Let $X$ be a Hopf order of $H$.
\begin{enumerate}
\item[{\it (i)}] The dual order $X^{\star}:=\{\varphi \in H^* : \varphi(X) \subseteq R\}$ is a Hopf order of $H^*$. \vspace{2pt}
\item[{\it (ii)}] The natural isomorphism $H \simeq H^{**}$ induces an isomorphism $X \simeq X^{\star \star}$ of Hopf orders. \vspace{2pt}
\end{enumerate}
\end{proposition}

The proofs of our main results are rooted in the following:

\begin{proposition}\cite[Proposition 1.2]{CM1} \label{character}
Let $X$ be a Hopf order of $H$. Then:
\begin{enumerate}
\item[{\it (i)}] Every character of $H$ belongs to $X^{\star}$. \vspace{2pt}
\item[{\it (ii)}] Every character of $H^*$ (cocharacter of $H$) belongs to $X$.
\end{enumerate}
\end{proposition}

The following technical result often eases our task:

\begin{proposition}\cite[Proposition 1.9]{CM1} \label{subsquo}
Let $X$ be a Hopf order of $H$.
\begin{enumerate}
\item[{\it (i)}] If $A$ is a Hopf subalgebra of $H$, then $X\cap A$ is a Hopf order of $A$. \vspace{2pt}
\item[{\it (ii)}] If $\pi:H \rightarrow B$ is a surjective Hopf algebra map, then $\pi(X)$ is a Hopf order of $B$.
\end{enumerate}
\end{proposition}

\subsection{Drinfeld twist}\label{drt}

We succinctly recall here the basics of Drinfeld's deformation procedure of a Hopf algebra. \vspace{1pt} An invertible element $\Omega:=\sum \Omega^{(1)} \otimes \Omega^{(2)} \in H \otimes H$ is called a {\it twist} for $H$ provided that:
$$\begin{array}{c}
(1_H \otimes \Omega)(id \otimes \Delta)(\Omega)=(\Omega \otimes 1_H)(\Delta \otimes id)(\Omega), \quad \textrm{and} \vspace{2pt} \\
(\varepsilon \otimes id)(\Omega)=(id  \otimes \varepsilon)(\Omega)=1_H.
\end{array}$$
The {\it Drinfeld twist} of $H$ is the new Hopf algebra $H_{\Omega}$ constructed as follows: $H_{\Omega}=H$ as an algebra, the counit is that of $H$, and the coproduct and antipode differ from those in $H$ in the following way:
$$\Delta_{\Omega}(h)={\Omega} \Delta(h){\Omega}^{-1} \qquad \textrm{and} \qquad S_{\Omega}(h)=Q_{\Omega}\hspace{1pt}S(h)\hspace{0.5pt}Q_{\Omega}^{-1} \qquad \forall h \in H.$$
Here, $Q_{\Omega}:=\sum {\Omega}^{(1)}S({\Omega}^{(2)})$ and $Q_{\Omega}^{-1}=\sum S({\Omega}^{-(1)})\Omega^{-(2)}.$ \vspace{1pt} For the latter we similarly write ${\Omega}^{-1}=\sum {\Omega}^{-(1)} \otimes {\Omega}^{-(2)}$. \vspace{-0.5mm}

\subsection{Twists of group algebras from 2-cocycles on abelian subgroups} \enlargethispage{1mm}
We next describe Movshev's method \cite[Section 1]{Mv} of constructing a twist for a group algebra from a $2$-cocycle on an abelian subgroup. \par \vspace{2pt}

Let $G$ be a finite group and $M$ an abelian subgroup of $G$. The group algebra $K\hspace{-0.9pt}M$ is a Hopf subalgebra of $K\hspace{-0.9pt}G$. Suppose that $\car K \nmid \vert M \vert$ and that $K$ is large enough for $K\hspace{-0.9pt}M$ to split. (Here and below, we use the term \emph{split} in the sense of \cite[Definition 7.12]{CR}: every\vspace{-1.5pt}  irreducible representation -corepresentation when dealing with coalgebras- is absolutely irreducible.) Consider the character group $\widehat{M}$ of $M$. The Wedderburn decomposition of $K\hspace{-0.9pt}M$ is provided by the complete set of orthogonal primitive\vspace{-1pt} idempotents $\{e_{\phi}\}_{\phi \in \widehat{M}},$ where $e_{\phi}$ is given by the formula:
$$e_{\phi} = \frac{1}{\vert M \vert } \sum_{v \in M} \phi(v^{-1}) v.$$
If $\omega: \widehat{M} \times \widehat{M} \rightarrow K^{\times}$ is a normalized 2-cocycle, then
$$\Omega_{M, \omega}:=\sum_{\phi,\psi \in \widehat{M}} \omega(\phi,\psi) e_{\phi} \otimes e_{\psi}$$
is a twist for $K\hspace{-0.9pt}M$, and, consequently, for $K\hspace{-0.9pt}G$. \par \vspace{2pt}

Later, we will need to know how a twist of this type is carried\vspace{1pt} under an automorphism $f:G \rightarrow G$ (in particular, under conjugation). We can\vspace{-1pt} carry $\phi$ to a character $\phi^f$ of $f(M)$ and $\omega$ to a normalized $2$-cocycle $\omega^f$ on $\widehat{f(M)}$ in the\vspace{-0.5pt} natural way: $\phi^f = \phi \circ (f\vert_M)^{-1}$ and $\omega^f = \omega \circ (\widehat{f\vert_M} \times \widehat{f\vert_M})$, respectively. For\vspace{1.5pt} the isomorphism $f:K\hspace{-0.9pt}G \mapsto K\hspace{-0.9pt}G, g \mapsto f(g),$ we have $f(e_{\phi})=e_{\phi^f}$ and $(f \otimes f)(\Omega_{M, \omega}) = \Omega_{f(M), \omega^f}$. Then:

\begin{remark}\label{twistiso}
The map $f:K\hspace{-0.9pt}G \mapsto K\hspace{-0.9pt}G$ establishes a Hopf algebra isomorphism between $(K\hspace{-0.9pt}G)_{\Omega_{M, \omega}}$ and $(K\hspace{-0.9pt}G)_{\Omega_{f(M), \omega^f}}$.
\end{remark}
\enlargethispage{\baselineskip}

It is shown in \cite[Corollary 3.6]{AEGN} that $(K\hspace{-0.9pt}G)_{\Omega_{M, \omega}}$ is cosemisimple. The Wedderburn decomposition of $(K\hspace{-0.9pt}G)_{\Omega_{M, \omega}}$ at the coalgebra level was described by Etingof and Gelaki in \cite[Section 3]{EG1}. We\vspace{1pt}  summarize \cite[Propositions 3.1, 4.1, and 4.2]{EG1} and \linebreak \cite[Propositions 2.1 and 2.2]{CM2} \hspace{2pt}in \hspace{1pt}the \hspace{1pt}following result. Recall that\vspace{-1pt} $\omega$ is called non-degenerate if\hspace{-1pt} the\hspace{-1pt} skew-symmetric\hspace{-1pt} pairing \hspace{-0.5pt}$\mathcal{B}_{\omega}\hspace{-1pt}:\hspace{-1pt}\widehat{M} \times \widehat{M} \rightarrow K^{\times}, (\phi,\psi) \mapsto \frac{\omega(\phi,\psi)}{\omega(\psi,\phi)},$ is\linebreak non-degenerate.

\enlargethispage{\baselineskip}
\begin{proposition}\label{decomp}
Let $\{\tau_{\ell}\}_{\ell=1}^n$ be a set of representatives of the double cosets of $M$ in $G$. Then:
\begin{enumerate}
\item[{\it (i)}] As a coalgebra, $(K\hspace{-0.9pt}G)_{\Omega_{M, \omega}}$ decomposes as the direct sum of subcoalgebras
$$(K\hspace{-0.9pt}G)_{\Omega_{M, \omega}} = \bigoplus_{\ell=1}^n\hspace{2pt} K(M\tau_{\ell} M).$$

\item[{\it (ii)}] Suppose\vspace{-2pt} that $K$ is large enough so that $(K\hspace{-0.9pt}G)_{\Omega_{M, \omega}}$ splits as a coalgebra. If $M \cap (\tau_{\ell} M \tau_{\ell}^{-1}) =\{1\}$ and $\omega$ is non-degenerate, then $K(M\tau_{\ell} M)$ is isomorphic to a matrix coalgebra of size $\vert M \vert.$ Moreover, the irreducible cocharacter of $(K\hspace{-0.9pt}G)_{\Omega_{M, \omega}}$ attached to $K(M\tau_{\ell} M)$ is $c_{\hspace{0.8pt}\tau_{\ell}}:=|M|e_{\varepsilon}\tau_{\ell}e_{\varepsilon}$.
\end{enumerate}
\end{proposition}

Recall from \cite[page 366]{Ka2} that $M$ is said to be of symmetric type if $M \simeq E \times E$ for some group $E$. We know by \cite[Theorems 1.9, 2.8, and 2.11]{Ka2} that $M$ admits a non-degenerate $2$-cocycle if and only if $M$ is of symmetric type. We will use the terminology \emph{central type} instead, which is the standard nowadays in this setting and applies to arbitrary groups, not necessarily abelian. See, for instance, the introductions of \cite{BG} and \cite{GS}. We stress that the property of being non-degenerate for a $2$-cocycle is preserved under multiplication by coboundaries; so it is indeed a property of its cohomology class.
\par \vspace{2pt}

The following proposition will be crucial in the proof of our main results:

\begin{proposition}\label{key}
Keep hypotheses and notation as in Proposition \ref{decomp}(ii). Set $\tau = \tau_{\ell}$ for short. Let $\chi:G \rightarrow K$ be a character. Then:
$$(\chi \otimes id)\Delta_{\Omega}(c_{\hspace{0.8pt}\tau})=\frac{1}{\vert M \vert}\sum_{g \in M \tau M} \chi(g)g.$$
\end{proposition}

\begin{proof}
The proof is a variation of that of \cite[Proposition 3.1(iii)]{CM2}. We compute:
$$\begin{array}{rl}
(\chi \otimes id)\Delta_{\Omega}(c_{\hspace{0.8pt}\tau}) & \hspace{-1mm} \overset{\text{\ding{172}}}{=} \hspace{1mm} {\displaystyle  \vert M \vert \sum_{\lambda,\hspace{0.5pt} \rho \in \widehat{M}} \omega(\lambda,\lambda^{-1})\omega^{-1}(\rho,\rho^{-1})\chi(e_{\lambda}\tau e_{\rho}) e_{\lambda^{-1}}\tau e_{\rho^{-1}}} \vspace{3pt} \\
 & \hspace{-1mm} \overset{\text{\ding{173}}}{=} \hspace{1mm} {\displaystyle  \vert M \vert \sum_{\lambda,\hspace{0.5pt}\rho \in \widehat{M}} \omega(\lambda,\lambda^{-1})\omega^{-1}(\rho,\rho^{-1})\chi(\tau e_{\rho}e_{\lambda}) e_{\lambda^{-1}}\tau e_{\rho^{-1}}} \vspace{3pt} \\
 & \hspace{-1mm} \overset{\text{\ding{174}}}{=} \hspace{1mm} {\displaystyle  \vert M \vert \sum_{\lambda \in \widehat{M}} \omega(\lambda,\lambda^{-1})\omega^{-1}(\lambda,\lambda^{-1})\chi(\tau e_{\lambda}) e_{\lambda^{-1}}\tau e_{\lambda^{-1}}} \vspace{3pt} \\
& \hspace{-1mm} \overset{\text{\ding{174}}}{=} \hspace{1mm} {\displaystyle  \vert M \vert \sum_{\lambda,\hspace{0.5pt}\rho \in \widehat{M}} \chi(\tau e_{\rho} e_{\lambda}) e_{\lambda^{-1}}\tau e_{\rho^{-1}}} \vspace{3pt} \\
& \hspace{-1mm} \overset{\text{\ding{173}}}{=} \hspace{1mm} {\displaystyle  \vert M \vert \sum_{\lambda,\hspace{0.5pt}\rho \in \widehat{M}} \chi(e_{\lambda} \tau e_{\rho} ) e_{\lambda^{-1}}\tau e_{\rho^{-1}}} \\
& \hspace{-1mm} \overset{\text{\ding{175}}}{=} \hspace{1mm} {\displaystyle (\chi \otimes id)\Delta(|M|e_{\varepsilon}\tau e_{\varepsilon})} \vspace{5pt} \\
& \hspace{-1mm} = \hspace{1mm} {\displaystyle \frac{\vert M \vert}{\hspace{4.5pt}\vert M \vert^2} \sum_{u,v \in M} \chi(u\tau v)u\tau v} \vspace{3pt} \\
& \hspace{-1mm} \overset{\text{\ding{176}}}{=} \hspace{1mm} {\displaystyle \frac{1}{\vert M \vert} \sum_{g \in M \tau M} \chi(g)g.}
\end{array}$$
Here, we used:
\begin{enumerate}
\item[\ding{172}] Definition of $\Delta_{\Omega}$ and Equation 2.2 in \cite[page 141]{CM2}. \vspace{1pt}
\item[\ding{173}] That $\chi$ is a character: $\chi(gh)=\chi(hg)$ for all $g,h \in G$. \vspace{1pt}
\item[\ding{174}] That $\{e_{\phi}\}_{\phi \in \widehat{M}}$ is a complete set of orthogonal idempotents in $K\hspace{-0.9pt}M$.  \vspace{1pt}
\item[\ding{175}] That $\Delta(e_{\phi})=\sum_{\lambda \in \widehat{M}} \hspace{2pt} e_{\lambda} \otimes e_{\lambda^{-1}\phi}$. \vspace{1pt}
\item[\ding{176}] That $u\tau v$ runs one-to-one all elements of $M\tau M$ since $\vert M \tau M \vert = \vert M \vert^2$.
\end{enumerate}
\end{proof}

\section{Strategy of proof and framework of application}\label{stfrm}

We expound here in general terms the strategy that we will use to prove the non-existence of integral Hopf orders for a twist of several group algebras of simple groups of Lie type. This strategy suitably modifies that employed in \cite[Section 3]{CM2} for the alternating groups. The modification consists in the single use of the induced character from a Sylow subgroup. This greatly widens the range of applications.

\subsection{Strategy}
Let $K$ be a number field and $R$ a Dedekind domain such that \linebreak $\Oint_K \subseteq R\subset K$. Consider the Hopf algebra $(K\hspace{-0.9pt}G)_{\Omega_{M, \omega}}$ as in the previous section. Suppose that $(K\hspace{-0.9pt}G)_{\Omega_{M, \omega}}$ admits a Hopf order $X$ over $R$. The aim will be to prove that for our choice of groups and twists the fraction $\frac{1}{\vert M \vert}$ belongs to $R$. Notice\vspace{-1pt} that if $\frac{1}{\vert M \vert} \in R$, then the group ring $RG$ is closed\vspace{-1pt} under $\Delta_{\Omega}, \varepsilon_{\Omega}$ and $S_{\Omega}$ and $RG$ is a Hopf order of $(K\hspace{-0.9pt}G)_{\Omega_{M, \omega}}$ over $R$. Our result will tell that the condition $\frac{1}{\vert M \vert} \in R$ is also necessary for $(K\hspace{-0.9pt}G)_{\Omega_{M, \omega}}$ to admit a Hopf order over $R$. This implies,\vspace{1.5pt} in particular, that Hopf orders of $(K\hspace{-0.9pt}G)_{\Omega_{M, \omega}}$ over $R$ do not exist when $R=\Ow_K$. \par \vspace{2pt}

For our aim, we can assume first, without loss of generality, that $\omega$ is non-degenerate. If $\omega$ is degenerate, we can\vspace{-0.2pt} replace $M$ by a subgroup $M_{r}$ of $M$ and $\omega$ by a non-degenerate $2$-cocycle $\omega_{r}$ on $\widehat{M_{r}}$ such\vspace{-1.2pt} that $\Omega_{M,\omega}=\Omega_{M_{r},\omega_{r}}$. This is achieved by factoring out $\widehat{M}$ by the radical of the skew-symmetric pairing\vspace{0.5pt} associated to $\omega$. We can assume secondly that $K$ is large enough so that $(K\hspace{-1pt}G)_{\Omega_{M,\omega}}$ splits\vspace{0.6pt} as an algebra and as a coalgebra. The reason for this is the following. The Hopf algebra $(K\hspace{-1pt}G)_{\Omega_{M,\omega}}$ splits at the level of both structures by a finite field extension $L/K$. Let $\bar{R}$ denote the integral closure of $R$ in $L$. Then, $\bar{R}$ is a Dedekind domain, which contains $\Oint_L$, and $X \otimes_R \bar{R}$ is a Hopf order of $(LG)_{\Omega_{M,\omega}}$ over $\bar{R}$. \par \vspace{2pt}

In our setting, $M$ will be a $p$-group and there will be a Sylow $p$-subgroup $U$ of $G$ set in advance. Remark \ref{twistiso} will allow us to assume that $M$ is a subgroup of $U$. We will work with the induced character $\chi:=\Ind_U^G(\Uno_U)$. Set $P=\bigcup_{g\in G} g \hspace{0.25pt} U\hspace{-1.25pt} g^{-1}$. By construction, $\chi$ vanishes outside $P$. Its values on the identity and on $u\in U$ are:
\begin{equation}\label{gchi}
\begin{array}{rl}
\chi(1) = & \hspace{-6pt} {\displaystyle \frac{|G|}{|U|}}, \vspace{6pt} \\
\chi(u) = & \hspace{-6pt} {\displaystyle \frac{1}{|U|}\sum_{\begin{subarray}{c} g \in G \\ gug^{-1}\in U \end{subarray}} \Uno_U (g ug^{-1}) = \frac{1}{|U|}\cdot \#\{g\in G : gug^{-1}\in U\}.}
\end{array}
\end{equation}

Next we will find $\tau \in G$ such that $M \cap (\tau M \tau^{-1}) =\{1\}$. We will consider the irreducible cocharacter $c_{\tau}$ of $(K\hspace{-0.9pt}G)_{\Omega_{M, \omega}}$ attached to $K(M\tau M)$. By Proposition \ref{character}(ii), $c_{\tau}$ belongs to $X$. Using that $X$ is closed under the coproduct, Proposition \ref{character}(i) and Proposition \ref{key}, we obtain that the following element $y_{\chi}$ belongs to $X$:
\begin{equation}\label{ychi}
y_{\chi}:= (\chi \otimes id)\Delta_{\Omega}(c_{\tau})=\frac{1}{\vert M \vert}\sum_{g \in M \tau M} \chi(g)g.
\end{equation}
Then, $y_{\chi}^2$ belongs to $X$ as well. Proposition \ref{character}(i) yields that $\chi(y_{\chi}^2) \in R$. The fraction $\frac{1}{\vert M \vert}$ will appear as a factor of $\chi(y_{\chi}^2)$ in a way that we could derive that $\frac{1}{\vert M \vert} \in R$ by applying Bezout's identity. \par \vspace{2pt}

The computation of $\chi(y_{\chi}^2)$ can be carried out in a systematic manner as follows. The element $y_\chi$ in Equation \ref{ychi} reads as:
\begin{align*}
y_\chi & = \frac{1}{|M|}\sum_{v,v'\in M} \chi(v \tau v') v \tau v'.
\end{align*}
Then,
\begin{equation}\label{ychi2}
y_\chi^2 = \frac{1}{|M|^2} \sum_{u,u'\hspace{-1pt},v,v' \in M} \chi(\tau u'u)\chi(\tau v'v) u\tau u'v\tau v'.
\end{equation}
Evaluating $\chi$ at this element, we get:
\begin{equation}\label{chiychi2}
\hspace{4.6mm}\chi(y_\chi^2) = \frac{1}{|M|^2} \sum_{u,u'\hspace{-1pt},v,v' \in M} \chi(\tau u'u)\chi(\tau v'v) \chi(\tau u'v \tau v'u).
\end{equation}

Let $I_\chi$ be\vspace{1pt} the set of fibers of $\chi$ of non-zero values; that is, $I_\chi=\{\chi^{-1}(\gamma) : \gamma \in \Ima(\chi) \textrm{ and } \gamma \neq 0\}$. For $C \in I_\chi$ we\vspace{1pt} write $\chi(C)$ for the value that $\chi$ takes at any element in $C$ and we\vspace{1.2pt} set $M_C=\{v \in M : \tau v \in C\}$. \hspace{-2pt}(In our setting we will have $\{1\}\hspace{-0.75pt}\in\hspace{-0.75pt} I_\chi$ and $M_{\{1\}}\hspace{0.1pt}=\hspace{0.3pt}\emptyset$.) Notice that $C \subset P$. \par \vspace{2pt}

We now rearrange the subscripts in the sums in \eqref{ychi2} and \eqref{chiychi2} as follows. Firstly, we make the substitutions $u'u=x \in M_C$ and $v'v=x' \in M_{C'}$ and eliminate $u'$ and $v$. Secondly, we replace $v'$ by $v^{-1}$. Thirdly, we rename $u^{-1}v$ as $v$ in \eqref{chiychi2}. Thus, we arrive at the following formulas for $y_\chi^2$ and $\chi(y_\chi^2)$:
\begin{equation}\label{ychi2fib}
\hspace{5mm}y_\chi^2 =\frac{1}{|M|^2} \sum_{C, C' \in I_\chi} \chi(C)\chi(C') \hspace{-2pt}\sum_{\begin{subarray}{c} u,v \in M \vspace{0.4pt} \\ x\in M_C,\, x'\in M_{C'} \end{subarray}} u\tau x x'u^{-1}v \tau v^{-1},\vspace{1mm}
\end{equation}
\begin{equation}\label{chiychi2fib}
\hspace{-5mm}\chi(y_\chi^2) = \frac{1}{|M|} \sum_{C, C' \in I_\chi} \chi(C)\chi(C') \hspace{-2pt}\sum_{\begin{subarray}{c} v \in M \vspace{0.8pt} \\ x \in M_C,\, x'\in M_{C'} \end{subarray}} \chi(\tau x x'v \tau v^{-1}).
\end{equation}

Then, the calculation of $\chi(y_\chi^2)$ reduces to the following procedure:
\begin{enumerate}
\item Find out $(\tau M) \cap P$ and $M_C$ for every fiber $C\neq\{1\}$ in $I_\chi$. \vspace{2pt}
\item Detect for which $v\in M, x\in M_C$, and $x'\in M_{C'}$, the element $\tau x x'v\tau v^{-1}$ belongs to $P$. \vspace{2pt}
\item For those elements obtained in step 2, calculate $\chi(\tau x x'v\tau v^{-1})$.
\end{enumerate}
\vspace{0.15mm}

\subsection{Application to finite groups of Lie type}\label{sec:strategy-p}

We now explain how to apply the previous strategy to finite groups of Lie type with defining characteristic $p$ and a twist arising from an
abelian $p$-subgroup of central type. All unexplained notions, properties, and results recalled here can be found in the monograph \cite{MT}. \par \vspace{2pt} \enlargethispage{\baselineskip}

Let $\G$ be a simply connected simple algebraic group defined over $\overbar{\mathbb{F}_{\hspace{-0.25pt} p}}$\hspace{0.75pt} and $F$ a Steinberg endomorphism; i.e., an endomorphism of the abstract group $\G$ such that the subgroup $\G^F$ consisting\vspace{0.75pt} of fixed points is finite. We will consider the group $\G^F$ and its central quotient $\G^F/Z(\G^F)$. With a few exceptions, $\G^F/Z(\G^F)$ is always simple. All simple groups of Lie type arise in this form, except for the Tits group, and only in very few cases the defining characteristic $p$ is not uniquely determined, see \cite[page 3]{W3}. For our purpose, any possible realization in this form would work. We\vspace{0.75pt} denote by $\pi\colon \G^F\to \G^F/Z(\G^F)$ the natural projection. Recall that $\gcd(p,|Z(\G^F)|)=1$. \par \vspace{2pt}

The group $\G$ contains an $F$-stable maximal torus $\T$ and two opposite $F$-stable unipotent subgroups $\U$ and $\U^-$ which are normalized\vspace{0.75pt} by $\T$ and satisfy $\U \cap\U^-=\{1\}$. The subgroups $\U^F$ and $(\U^-)^F$ are Sylow $p$-subgroups of $\G^F$. We denote them by $U$ and $U^-$ respectively. The quotient $\pi(U)$, which is\vspace{0.75pt} isomorphic to $U$, is in turn a Sylow $p$-subgroup of $\G^F/Z(\G^F)$. The same holds for $U^-$. The groups $\B:=\T\ltimes\U$ and $\B^-:=\T\ltimes\U^-$ are opposite\vspace{0.75pt} Borel subgroups of $\G$. As before, we write $B=\B^F=\T^F\ltimes U$ and $B^-=(\B^-)^F=\T^F\ltimes U^-$. Then, $B\cap B^-=\T^F$. Finally, recall that there is an element $\dot{w}_0\in N_{\G^F}(\T)$ such that $\dot{w_0}\U\dot{w}_0^{-1}=\U^-$, so $\dot{w}_0U\dot{w}_0^{-1}=U^-$. For any $\sigma\in\dot{w}_0\T^F$ we have $\sigma  U\sigma^{-1}=U^-$. The coset $\dot{w}_0\T^F$ is an involution in $N_{\G^F}(\T)/\T^F$. Hence, $\sigma^{-1}\in\sigma\T^F$. \par \vspace{2pt}

Assume that $U$ (or equivalently, $\pi(U)$) contains an abelian subgroup $V$ of central type. Observe that $V \cap (\sigma V \sigma^{-1})=\{1\}$ for any $\sigma \in \dot{w}_0\T^F$ (or $\pi(\sigma)$), as that is a subset of $U \cap U^-$. The double cosets $V\sigma V$ and $V$ are disjoint. Then, $v \sigma v'\neq 1$ for all $v,\,v' \in V$. Similarly, $\pi(v\sigma v') \neq 1$, since, otherwise, we would have $\sigma\in UZ(\G^F)$, which is impossible because $UZ(\G^F)$ normalizes $U$. \par \vspace{2pt}

The next section deals mainly with families of groups in the following two scenarios:
\begin{enumerate}
\item The group $\G^F$, the subgroup $M=V$, the element $\tau=\sigma$, and the induced character $\chi={\rm Ind}_{U}^{\G^F}\hspace{-2pt}(\Uno_U)$. \vspace{2pt}
\item The group $\G^F/Z(\G^F)$, the subgroup $M=\pi(V)$, the element $\tau=\pi(\sigma)$, and the induced character $\chi={\rm Ind}_{\pi(U)}^{G}(\Uno_{\pi(U)})$.
\end{enumerate}

The following remark identifies when an element in $P$ of the form $\tau u' v \tau v'u$ is the identity element. This is necessary in practice for the evaluation of $\chi$, see Equation \ref{chiychi2}:

\begin{remark}\label{rem:identity}
\hspace{-1.1pt}In the above scenarios, observe that $\tau u' v \tau v'u=1$ if and only if $u'v=1$, $v'u=1$, and $\tau=\tau^{-1}$. For, suppose that $\tau u' v \tau v'u=1$. Then, $\tau u' v \tau=(v'u)^{-1}$ belongs to $M \cap (\tau M \tau)$. In the first scenario, when dealing with $\G^F$, we have the inclusions $M \cap (\tau M \tau) \subseteq M \cap \scalebox{1.12}{(}\tau M (\tau^{-1} \T^F)\scalebox{1.12}{)} \subseteq U \cap \scalebox{1.12}{(}U^-\hspace{0.5pt}\T^F\scalebox{1.12}{)}=\{1\}$. Hence, $\tau u' v \tau =1$ and $v'u=1$. This implies in turn that $\tau^2 \in M \cap (\tau M \tau).$ Consequently, $\tau^2=1$ and $u'v=1$. For the proof in the second scenario, when dealing with $\G^F/Z(\G^F)$, use in addition that $\scalebox{1.12}{(}U Z(\G^F)\scalebox{1.12}{)} \cap \scalebox{1.12}{(}U^-\hspace{0.5pt}\T^F\scalebox{1.12}{)}=Z(\G^F) \cap \T^F.$
\end{remark}

\section{Twists of several quasisimple groups of Lie type arising from a $p$-subgroup}
\enlargethispage{2.1\baselineskip}

In this section we apply the strategy expounded before to several\vspace{0.75pt} groups, which are either of the form $\G^F$ or $\G^F/Z(\G^F)$, for $\G$ a matrix group. Here, $U$ will be a group of unipotent upper or lower triangular matrices, and conjugation by $\sigma$ will interchange upper and lower triangular matrices.

\subsection{Statement}

The aim of this section is to establish the following result:

\begin{theorem}\label{main1}
Let $K$ be a number field and $R\subset K$ a Dedekind domain such that $\Oint_K \subseteq R$. Let
\begin{enumerate}
\item[$\sq$] $p$ be a prime number and $q=p^m$ with $m>1$. \vspace{2pt}
\item[$\sq$] $G$ be one of the following finite quasisimple groups: $\SL_2(q), \PSL_2(q), \SL_3(p)$ or the Suzuki group $^2\!B_2(q)$ (here, $p=2$ and $m$ is odd). \vspace{2pt}
\item[$\sq$] $M$ be any abelian $p$-subgroup of central type. \vspace{1pt}
\item[$\sq$] $\omega: \widehat{M} \times \widehat{M} \rightarrow K^{\times}$ be any normalized non-degenerate cocycle.
\end{enumerate}
If $(K\hspace{-1pt}G)_{\Omega_{M,\omega}}$ admits a Hopf order over $R$, then $\frac{1}{\vert M \vert} \in R$. Hence,
$(K\hspace{-1pt}G)_{\Omega_{M,\omega}}$ does not admit a Hopf order over $\Oint_K$. \par \vspace{2pt}
\end{theorem}

\begin{proof}
The proof is carried out in the following subsections.
\end{proof}

Before tackling the proof of Theorem \ref{main1}, we record the following consequence. Its proof is similar to that of \cite[Corollary 2.4]{CM1}. One only has to add the fact that, up to coboundaries, $\omega$ can be chosen in such a way that its image is contained in a cyclotomic number field, see \cite[Proof of Proposition 2.1.1]{Ka1}.

\begin{corollary}\label{complexif}
Keep the hypotheses on $(G,M,\omega)$ of Theorem \ref{main1}. Then, the complex semisimple Hopf algebra $(\Co G)_{\Omega_{M,\omega}}$ does not admit a Hopf order over any number ring. \vspace{0.5mm}
\end{corollary}

\subsection{Special linear group $\SL_2(q)$}\label{subsec:SL2}
\enlargethispage{\baselineskip}

Let $p$ be a prime number and $q=p^m$, with $m>1$. Consider the group $G=\SL_2(q)$. The subgroup $U:=\left\{{\left(\begin{smallmatrix} 1 & a \\ 0 & 1\end{smallmatrix}\right)} : a \in \F_q\right\}$ is a Sylow $p$-subgroup of $G$. It is isomorphic to the elementary $p$-group $C_p^m$. \par \vspace{2pt}

Let $M$ be an abelian $p$-subgroup of $G$ of central type. By Sylow's theorem and Remark \ref{twistiso}, we can assume, without loss of generality, that $M$ is contained in $U$. The group of diagonal matrices in $\GL_2(q)$ acts by automorphisms on $G$, on $U$, and on the set of subgroups of $U$ of central type. The orbit of any such subgroup has a representative containing the matrix $\left(\begin{smallmatrix} 1 & 1 \\ 0 & 1\end{smallmatrix}\right)$. Let $M$ be one of these representatives. Then, $M=\left\{{\left(\begin{smallmatrix} 1 & a \\
 0 & 1\end{smallmatrix}\right)}: a \in \mathbb{E}\right\}$, where $\mathbb{E}$ is an additive subgroup of $\F_q$ isomorphic to $C_p^{2n}$ for some $n>0$. Our choice of $M$ ensures that $\F_p \subset \mathbb{E}$. Remark \ref{twistiso} allows us to make this choice for our purpose. \par \vspace{2pt}

We set $T=\left\{{\left(\begin{smallmatrix} t & 0 \\
0 & t^{-1}\end{smallmatrix}\right)} : t \in \F_q^\times\right\}$ and $B=T\ltimes U$. Furthermore,\vspace{1pt} we pick the element $\tau=\sigma=\left(\begin{smallmatrix} 0 & -1 \\ 1 & 0\end{smallmatrix}\right)$ in $N_G(T)$. Using Equation \ref{gchi}, one can check\vspace{1pt} that the non-zero values of $\chi={\rm Ind}_{U}^{G}(\Uno_U)$ are:
$$\chi(1)= q^2-1 \hspace{6mm} \textrm{and} \hspace{6mm} \chi \hspace{-1.5pt}\left(\begin{smallmatrix} 1 & a \\
0 & 1\end{smallmatrix}\right) = q-1, \hspace{3mm} \textrm{for } a \neq 0.\vspace{2pt}$$

In this case, the sets $P$ and $I_\chi$ particularize to $P=\{A\in\SL_2(q): \Tr(A)=2\}$ and $I_\chi=\{\{1\},P\menos\{1\}\}$.
Let $P^{\bullet}=P\menos\{1\}$. A direct calculation shows that $M_{P^{\bullet}}=\{1\}$ if $p=2$ and $M_{P^{\bullet}}=\left\{\left(\begin{smallmatrix}
1&2\\
0&1\end{smallmatrix}\right)\right\}$ otherwise. (We stress that $\left(\begin{smallmatrix} 1 & 2 \\ 0 & 1\end{smallmatrix}\right) \in M$ thanks to our choice of $M$.) Equation \ref{chiychi2fib} takes the following concrete form:
$$\begin{array}{l}
\chi(y_\chi^2) = {\displaystyle \frac{\chi(P^{\bullet})^2}{|M|} \sum_{v \in M} \chi
\hspace{-2pt}\left(\tau \hspace{-1pt}\left(\begin{smallmatrix} 1 & 2 \\ 0 & 1\end{smallmatrix}\right)^2 \hspace{-1.5pt} v\tau v^{-1}\right)} \vspace{7pt} \\
\phantom{\chi(y_\chi^2)} = {\displaystyle  \frac{(q-1)^2}{p^{2n}} \hspace{1pt} \sum_{a \in \E} \chi \Big(
 \hspace{-1.5pt}\left(\begin{smallmatrix} 0 & -1 \\ 1 & 0 \end{smallmatrix}\hspace{-0.1mm}\right)
 \hspace{-1.5pt}\left(\begin{smallmatrix} 1 & \hspace{-3.75mm}\phantom{-1}4 \\ 0 & 1 \end{smallmatrix}\hspace{-0.1mm}\right)
 \hspace{-1.5pt}\left(\begin{smallmatrix} 1 & \hspace{-3.75mm}\phantom{-1}a \\ 0 & 1\end{smallmatrix}\right)
 \hspace{-1.5pt}\left(\begin{smallmatrix} 0 & -1 \\ 1 & 0\end{smallmatrix}\hspace{-0.1mm}\right)
 \hspace{-1.5pt}\left(\begin{smallmatrix} 1 & -a \\ 0 & 1\end{smallmatrix}\right)\hspace{-1.5pt}
 \Big)} \vspace{7pt} \\
\phantom{\chi(y_\chi^2)} = {\displaystyle \frac{(q-1)^2}{p^{2n}} \hspace{1pt}\sum_{a \in \E} \chi \hspace{-1.5pt}\left(\begin{smallmatrix} -1 & \hspace{3pt}a \\
4+a & \hspace{3pt}-1-4a-a^2\end{smallmatrix}\right).}
\end{array}$$
Observe that $\left(\begin{smallmatrix} -1 & a \\ 4+a & \hspace{2pt}-1-4a-a^2\end{smallmatrix}\right)\in P$ if and only if $a=-2$.\vspace{-1pt} Then, the only non-zero term in this sum corresponds to $a=-2$. Its value\vspace{2pt} is $q^2-1$ if $p=2$ and $q-1$ otherwise. Hence:
$$\chi(y_\chi^2) = \frac{(q-1)^3(q+1)}{p^{2n}} \hspace{1.5mm} \textrm{ if } p=2 \hspace{5mm} \textrm{ and } \hspace{5mm}
\chi(y_\chi^2) = \frac{(q-1)^3}{p^{2n}} \hspace{1.5mm} \textrm{ if } p \neq 2.$$
In both cases, $\chi(y_\chi^2)$ is an irreducible fraction. Propositions \ref{character} and \ref{decomp}, together with Bezout's identity, yield $\frac{1}{p} \in R$. \par \vspace{2pt}

This finishes the proof of Theorem \ref{main1} for $\SL_2(q)$. \qed
\vspace{0.5mm}

\subsection{Projective special linear group $\PSL_2(q)$}\label{PSL(2,q)}

In this subsection we assume that $p$ is odd, since $\PSL_2(q)=\SL_2(q)$ when $p=2$ and this case was just treated. We denote by $\pi:\SL_2(q) \rightarrow \PSL_2(q)$ the natural projection. We retain the necessary notation from the preceding subsection. \par \vspace{2pt}

Our group $G$ is now $\PSL_2(q)$. Let $M$ be an abelian $p$-subgroup of $G$ of central type. Consider the subgroup $\pi(U)$, which is a Sylow $p$-subgroup of $G$. By Sylow's theorem and Remark \ref{twistiso}, we can assume that $M$ is contained in $\pi(U)$. An abelian subgroup $M$ of $\pi(U)$ of central type is of the form $\pi(N)$, where $N$ is abelian subgroup of $U$ of central type. Arguing as before with the action on $\PSL_2(q)$ induced by conjugation by a diagonal matrix in $\GL_2(q)$, we can pick a subgroup $M$ of square order $p^{2n}$ that contains $\pi \hspace{-2pt}\left(\begin{smallmatrix}1&1\\ 0&1\end{smallmatrix}\right)$ and is of the form $\left\{{\pi\hspace{-2pt}\left(\begin{smallmatrix} 1 & a \\ 0 & 1\end{smallmatrix}\right)}: a \in \mathbb{E}\right\}$, with $\mathbb{E}$ as above. Remark \ref{twistiso} allows us to reduce the proof to this case. \par \vspace{2pt}

Our element $\tau$ is now $\pi \hspace{-2pt} \left(\begin{smallmatrix} 0 & -1 \\ 1 & 0\end{smallmatrix}\right).$  We take $\chi=\Ind_{\pi(U)}^{G}(\Uno_{\pi(U)})$. One can check that the non-zero values of $\chi$ are:
$$\chi(1)= \frac{q^2-1}{2} \hspace{6.5mm} \textrm{and} \hspace{6.5mm} \chi \hspace{-1pt}\left( \pi \hspace{-2pt}\left(\begin{smallmatrix} 1 & a \\
0 & 1\end{smallmatrix}\right)\right) = \frac{q-1}{2}, \hspace{3mm} \textrm{for }a \neq 0.$$

One\vspace{1.25pt} can also verify that, in this case, $P=\{\pi(A) \in \PSL_2(q) : \Tr(A)=\pm 2\}$. Then, $I_\chi=\{\{1\}, P\menos\{1\}\}$. Set, as before, $P^{\bullet}=P\menos\{1\}$. A\vspace{1.5pt} direct calculation shows that $M_{P^{\bullet}}=\left\{\pi \hspace{-2pt}\left(\begin{smallmatrix} 1 & \pm 2\\ 0 & 1\end{smallmatrix}\right)\right\}$. Equation \ref{chiychi2fib} now reads as:
\begin{equation}\label{PSL2sum}
\begin{array}{ll}
\chi(y_\chi^2) = & \hspace{-2mm} {\displaystyle \frac{(q-1)^2}{4|M|}\left(2\sum_{v\in M}\chi \Big(\pi \hspace{-2pt}\left(\tau v\tau  v^{-1}\right)\hspace{-1.5pt}\Big)+\sum_{v\in M}\chi \Big(\pi \hspace{-2pt}\left(\tau \hspace{-1pt}\left(\begin{smallmatrix} 1 & 4 \\ 0 & 1\end{smallmatrix}\right)\hspace{-1pt} v\tau v^{-1}\right)\hspace{-1.5pt}\Big)\right.} \hspace{5mm} \\
 & \hspace{26mm} {\displaystyle \left.+\sum_{v\in M}\chi \Big(\pi \hspace{-2pt}\left(\tau \hspace{-1pt}\left(\begin{smallmatrix} 1 & -4 \\ 0 & 1\end{smallmatrix}\right)\hspace{-1pt} v \tau v^{-1}\right) \hspace{-1.5pt}\Big)\right).}
\end{array}
\end{equation}
We compute the value of the three summands between parentheses: \par \vspace{2pt}

Firstly, an element of the form
$$\pi \hspace{-2pt} \left(
\hspace{-1pt}\left(\begin{smallmatrix} 0 & -1 \\ 1 & 0\end{smallmatrix}\hspace{-0.1mm}\right)
\hspace{-1pt}\left(\begin{smallmatrix} 1 & \hspace{-3.75mm}\phantom{-1}a \\ 0 & 1\end{smallmatrix}\right)
\hspace{-1pt}\left(\begin{smallmatrix} 0 & -1 \\ 1 & 0\end{smallmatrix}\hspace{-0.1mm}\right)
\hspace{-1pt}\left(\begin{smallmatrix} 1 & \hspace{-3.75mm}\phantom{-1}a \\ 0 & 1\end{smallmatrix}\right)^{-1}
\right)
= \pi \hspace{-2pt}\left(\begin{smallmatrix}
-1 & a \\ a & -1-a^2 \end{smallmatrix}
\right)$$
belongs to $P$ if and only if $a=0$ (i.e., it is the identity element) or $a^2=-4$. The latter occurs if and only if $\E$ contains a square root of $-4$. In such a case, both roots are in $\E$ and the corresponding elements are non-trivial and distinct. Hence, the value of the first summand is $2\chi(1)=q^2-1$ if $\sqrt{-4} \notin \E$ or $2\hspace{1pt} \big(\chi(1)+2\chi \hspace{-1pt}\left(\pi \hspace{-2pt} \left(\begin{smallmatrix} 1 & 1 \\ 0 & 1 \end{smallmatrix}\right)\right)\hspace{-1pt}\big)=q^2+2q-3$ otherwise. \par \vspace{2pt}

Secondly, an element of the form
$$\pi \hspace{-2pt} \left(
\left(\begin{smallmatrix} 0 & -1 \\ 1 & 0\end{smallmatrix}\hspace{-0.1mm}\right)
\hspace{-1pt}\left(\begin{smallmatrix} 1 & \hspace{-3.75mm}\phantom{-1}4 \\ 0 & 1\end{smallmatrix}\right)
\hspace{-1pt}\left(\begin{smallmatrix} 1 & \hspace{-3.75mm}\phantom{-1}a \\ 0 & 1\end{smallmatrix}\right)
\hspace{-1pt}\left(\begin{smallmatrix} 0 & -1 \\ 1 & 0\end{smallmatrix}\hspace{-0.1mm}\right)
\hspace{-1pt}\left(\begin{smallmatrix} 1 & \hspace{-3.75mm}\phantom{-1}a \\ 0 & 1\end{smallmatrix}\right)^{-1}
\right)
= \pi \hspace{-2pt} \left(\begin{smallmatrix} -1 & a \\ 4+a  & \hspace{2pt} -1-4a-a^2\end{smallmatrix}
\right)$$
belongs to $P$ if and only if $a \in\{0,-2,-4\}$. The three elements obtained with these values of $a$ are all non-trivial and distinct. Hence, the second summand equals $\frac{3(q-1)}{2}$. \par \vspace{2pt}

Finally, and in a similar fashion, an element of the form
$$\pi \hspace{-2pt} \left(
\left(\begin{smallmatrix} 0 & -1 \\ 1 & 0\end{smallmatrix}\hspace{-0.1mm}\right)
\hspace{-1pt}\left(\begin{smallmatrix} 1 & -4 \\ 0 & 1\end{smallmatrix}\right)
\hspace{-1pt}\left(\begin{smallmatrix} 1 & \hspace{-3.75mm}\phantom{-1}a \\ 0 & 1\end{smallmatrix}\right)
\hspace{-1pt}\left(\begin{smallmatrix} 0 & -1 \\ 1 & 0\end{smallmatrix}\hspace{-0.1mm}\right)
\hspace{-1pt}\left(\begin{smallmatrix} 1 & \hspace{-3.75mm}\phantom{-1}a \\ 0 & 1\end{smallmatrix}\right)^{-1}
\right)
= \pi \hspace{-2pt} \left(\begin{smallmatrix} -1 & a \\ -4+a  & \hspace{2pt} -1+4a-a^2\end{smallmatrix}
\right)$$
belongs to $P$ if and only if $a \in\{0,2,4\}$. The elements obtained with these three values of $a$ are all non-trivial and distinct. The third summand equals $\frac{3(q-1)}{2}$ as well. \par \vspace{2pt}

In total, the value of the sum in \eqref{PSL2sum} is $q^2+3q-4$ if $\sqrt{-4} \notin \E$ and $q^2+5q-6$ otherwise. Therefore, we have:
$$\chi(y_\chi^2) = \frac{(q-1)^3}{4p^{2n}}(q+4) \hspace{1.5mm} \textrm{ if } \sqrt{-4} \notin \E \hspace{2mm} \textrm{ and } \hspace{2mm}
\chi(y_\chi^2) = \frac{(q-1)^3}{4p^{2n}}(q+6) \hspace{1.5mm} \textrm{ if } \sqrt{-4} \in \E.$$
In both cases, $\chi(y_\chi^2)$ belongs to $\Q \menos \Z$. It follows from this and Propositions \ref{character} and \ref{decomp} that $\frac{1}{p} \in R$. \par \vspace{2pt}

This finishes the proof of Theorem \ref{main1} for $\PSL_2(q)$. \qed

\begin{remark}
For an abelian $p$-subgroup $M$ of $\SL_2(q)$ of central type, $\pi \vert_M:M \rightarrow \pi(M)$ is an isomorphism. The natural projection induces a surjective Hopf algebra map $\pi:(K\SL_2(q))_{\Omega_{M,\omega}} \rightarrow (K\PSL_2(q))_{\Omega_{\pi(M),\omega^{\pi}}}$. When $q$ is odd, the statement for $\SL_2(q)$ follows from that for $\PSL_2(q)$ in virtue of Proposition \ref{subsquo}(ii).
\end{remark}

\subsection{Special linear group $\SL_3(p)$}\label{subsec:SL3}
We will deduce the desired statement from a slightly more general result. Let $p$ be a prime number\vspace{0.75pt} and $q=p^m$, with $m \geq 1$. We will initially work with $G=\SL_3(q)$. Consider the Sylow $p$-subgroup
$$U:=\left\{\left(\begin{smallmatrix}
1 & a & b \\
0 & 1 & c \\
0 & 0 & 1
\end{smallmatrix}\right) : a,b,c \in \F_q\right\}$$
and the maximal split torus
$$T:=\left\{\left(\begin{smallmatrix}
s & \hspace{3pt}0 & \hspace{2pt}0 \\
0 & \hspace{3pt}t & \hspace{2pt}0 \\
0 & \hspace{3pt}0 & \hspace{2pt}t^{-1}s^{-1}
\end{smallmatrix}\right) : s, t \in \F_q^\times\right\}.$$
Bear in mind that $N_G(U)=B=T\ltimes U$. \par \vspace{2pt}

We first calculate the values at $U$ of the character $\chi={\rm Ind}_{U}^{G}(\Uno_U)$. Every element in $U$ has a Jordan canonical form. There are two possible non-trivial canonical forms: $\left(\begin{smallmatrix}
1 & 1 & 0 \\
0 & 1 & 0 \\
0 & 0 & 1
\end{smallmatrix}\right)$ (Jordan type $(2,1)$) and $\left(\begin{smallmatrix}
1 & 1 & 0 \\
0 & 1 & 1 \\
0 & 0 & 1
\end{smallmatrix}\right)$ (Jordan type $(3)$).

\begin{lemma}\label{chiSLq}
For $u \in U$ we have:
$$\chi(u) = \left\{\begin{array}{l}
(q^3-1)(q^2-1) \hspace{2mm} \textrm{ if $u=1$,} \vspace{2pt} \\
(2q+1)(q-1)^2 \hspace{2mm} \textrm{ if $u$ is of Jordan type $(2,1)$,} \vspace{2pt} \\
(q-1)^2 \hspace{3mm} \textrm{if $u$ is of Jordan type $(3)$.}
\end{array}\right.$$
\end{lemma}
\enlargethispage{1.4\baselineskip}

\begin{proof}
The value at the identity element is straightforward. The proof for the other two values is divided into three steps:\par \vspace{2pt}

Step 1. Let $u,u' \in U$ be such that $u'=kuk^{-1}$ for some $k \in \GL_3(q)$. We show that $\chi(u')=\chi(u)$. We write $k=e_kd_k$, with $e_k \in G$ and $d_k \in \GL_3(q)$ diagonal. Notice that conjugation by a diagonal matrix in $\GL_3(q)$ stabilizes $U$. Since $\chi$ is a character of $G$, we get:
$$\chi(u')=\chi(e_kd_kud_k^{-1}e_k^{-1})=\chi(d_kud_k^{-1}).$$
We now compute $\chi(u')$ and $\chi(u)$ by using Equation \ref{gchi}:
$$\begin{array}{l}
\chi(u') = {\displaystyle \#\{g \in G : g(d_kud_k^{-1})g^{-1}\in U\} \cdot |U|^{-1}} \vspace{5pt} \\
\phantom{\chi(u')} = {\displaystyle \#\{g\in G : d_k^{-1}gd_kud_k^{-1}g^{-1}d_k \in U\} \cdot |U|^{-1}} \hspace{6mm} \textrm{(Put $h=d_k^{-1}gd_k$)} \vspace{5pt} \\
\phantom{\chi(u')} = {\displaystyle \#\{h \in G : huh^{-1} \in U\} \cdot |U|^{-1}} \vspace{5pt} \\
\phantom{\chi(u')} = \chi(u).
\end{array}$$

In view\vspace{2pt} of the preceding statement, $\chi(u)=\chi(u_{_{\hspace{-1pt}J\hspace{-0.5pt}F}})$, where $u_{_{\hspace{-1pt}J\hspace{-0.5pt}F}}$ is the Jordan canonical form of $u$. In the next\vspace{2pt} steps we calculate the value of $\chi$ at the two possible Jordan canonical forms. \par \vspace{2pt}

Step 2. Set $u=\left(\begin{smallmatrix}
1 & 1 & 0 \\
0 & 1 & 0 \\
0 & 0 & 1
\end{smallmatrix}\right)$. We claim that $\chi(u)=(2q+1)(q-1)^2$. Let $C(u)$ and $Cl(u)$ denote\vspace{2.5pt} the centralizer and the conjugacy class, respectively, of $u$ in $G$. Consider the bijective\vspace{3pt} map $f:G/C(u) \rightarrow Cl(u),\ gC(u) \mapsto gug^{-1}$. Taking the inverse image of $Cl(u) \cap U$ under $f$, we obtain the following equality:
$$\# \{g \in G : gug^{-1} \in U\} = \vert C(u) \vert \vert Cl(u) \cap U \vert.$$
One can check in a direct way, by computing explicitly $C(u)$, that $\vert C(u) \vert = (q-1)q^3$. On the other hand, the Jordan type of a non-trivial element $w$ in $U$ can be detected by calculating the rank of $w-1$: rank $1$ corresponds to type $(2,1)$ and rank $2$ to type $(3)$. The set $Cl(u) \cap U$ consists of elements $w \in U$ that are conjugate to $u$. It can be shown that these are precisely the  elements $w \in U$ such that $\textrm{rk}(w-1)=1$. There are $(2q+1)(q-1)$ matrices fulfilling this condition. Equation \ref{gchi} now applies. \par \vspace{2pt}

Step 3. Finally,\vspace{0.3pt} set $u=\left(\begin{smallmatrix}
1 & 1 & 0 \\
0 & 1 & 1 \\
0 & 0 & 1
\end{smallmatrix}\right)$. We claim that $\chi(u)=(q-1)^2$. For, one first establishes\vspace{3.5pt} the equality
$\{g \in G : gug^{-1} \in U\} = T \ltimes U$ by direct manipulation with matrices and then applies Equation \ref{gchi}.
\end{proof}

We can rephrase $P$ as the set consisting of those matrices in $G$ whose characteristic polynomial equals $(1-z)^3$. We already know that $\chi$ vanishes outside $P$. The value of $\chi$ at an element $g \in P$ is obtained by determining its Jordan type through the calculation of  $\textrm{rk}(g-1)$ and then applying Lemma \ref{chiSLq}. In this case, the set $I_{\chi}$ is larger than in other cases: namely, $I_{\chi}=\{C_{(1)},C_{(2,1)},C_{(3)}\}$, where $C_{(1)}=\{1\}$, and $C_{(2,1)}$ and $C_{(3)}$ are the subsets of $P$ consisting of matrices of type $(2,1)$ and of type $(3)$ respectively. \par \vspace{2pt}

We are now in a position to prove the following result:

\begin{proposition}\label{propSL3q}
The conclusion of Theorem \ref{main1} holds for $\SL_3(q)$ and the following abelian subgroups of central type contained in $U$:
\begin{enumerate}
\item[(i)] Case $q$ even:
$$M_1=\left\{\left(\begin{smallmatrix} 1 & a & b \\ 0 & 1 & 0 \\ 0 & 0 & 1 \end{smallmatrix}\right) : a, b \in \F_q\right\}.$$
\item[(ii)] Case $q$ odd:
$$M_1=\left\{\left(\begin{smallmatrix} 1 & a & b \\ 0 & 1 & 0 \\ 0 & 0 & 1 \end{smallmatrix}\right) : a, b \in \F_q\right\} \quad \textrm{and}  \quad M_2=\left\{\left(\begin{smallmatrix} 1 & a & b \\ 0 & 1 & a \\ 0 & 0 & 1 \end{smallmatrix}\right) : a, b \in \F_q\right\}.$$
\end{enumerate}
\end{proposition}

\begin{proof}
We first proceed with $M_1$. We\vspace{1pt} pick the element $\tau=\left(\begin{smallmatrix} 0 & 1 & 0 \\ 0 & 0 & 1 \\ 1 & 0 & 0 \end{smallmatrix}\right)$ in $N_G(T)$. (Notice that, unlike other cases in this paper, this\vspace{2.5pt} choice of $\tau$ is not based on the element $\dot{w}_0$ mentioned in Subsection \ref{sec:strategy-p}.) It is\vspace{2.5pt} easy (and now necessary) to check that $M_1 \cap \big(\tau M_1 \tau^{-1}\big)=\{1\}$. We must find those elements $v \in M_1$ such that the\vspace{0.5pt} characteristic polynomial of $\tau v$ is $(1-z)^3$. This only\vspace{1pt} happens for $x:=\left(\begin{smallmatrix}
1 & -3 & 3 \\
0 &  1 & 0 \\
0 &  0 & 1
\end{smallmatrix}\right)$. The Jordan type of $\tau x$ is $(3)$. Equation \ref{chiychi2fib} can now be read as follows:
$$\begin{array}{l}
\chi(y_\chi^2) = {\displaystyle  \frac{\chi(C_{(3)})^2}{|M_1|} \sum_{v \in M_1} \chi\hspace{-1.5pt}\left(\tau x^2 v\tau v^{-1}\right)} \vspace{7pt} \\
\phantom{\chi(y_\chi^2)} = {\displaystyle \frac{(q-1)^4}{q^2} \sum_{a,b \in \F_q} \chi
\hspace{-1.5pt}\left(
\left(\begin{smallmatrix} 0 & 1  & 0 \\ 0 & 0 & 1 \\ 1 & 0 & 0 \end{smallmatrix}\right)
\hspace{-2pt}\left(\begin{smallmatrix} 1 & -3 & 3 \\ 0 &  1 & 0 \\ 0 &  0 & 1 \end{smallmatrix}\right)^{\hspace{-1pt} 2}
\hspace{-2pt}\left(\begin{smallmatrix} 1 &  a & b \\ 0 &  1 & 0 \\ 0 &  0 & 1 \end{smallmatrix}\right)
\hspace{-2pt}\left(\begin{smallmatrix} 0 & 1  & 0 \\ 0 & 0 & 1 \\ 1 & 0 & 0 \end{smallmatrix}\right)
\hspace{-2pt}\left(\begin{smallmatrix} 1 &  -a & -b \\ 0 &  1 & 0 \\ 0 &  0 & 1 \end{smallmatrix}\right)
\right)} \vspace{5pt} \\
\phantom{\chi(y_\chi^2)} = {\displaystyle \frac{(q-1)^4}{q^2}\sum_{a,b \in \F_q} \chi \hspace{-1.5pt}\left(\begin{smallmatrix}
0   & \hspace{6pt} 0        & \hspace{6pt}1 \\
1   & \hspace{6pt}-a        & \hspace{6pt}-b \\
b+6 & \hspace{6pt} 1-a(b+6) & \hspace{6pt}(a-6)-b(b+6)
\end{smallmatrix}\right).}
\end{array}$$
To continue this calculation, we need to know when the latter generic matrix belongs to $P$. By computing its characteristic polynomial, one can verify that this happens if and only if $a=-b=3$. The matrix obtained with these values is of type $(3)$ if $p \neq 2$ and of type $(2,1)$ if $p=2$. Therefore:
$$\chi(y_\chi^2) = \frac{(q-1)^6}{q^2} \hspace{1.5mm} \textrm{ if } p \neq 2 \hspace{5mm} \textrm{ and } \hspace{5mm} \chi(y_\chi^2) = \frac{(q-1)^6(2q+1)}{q^2} \hspace{1.5mm} \textrm{ if } p=2.$$
From this, it follows that $\frac{1}{p} \in R$. This proves the statement for $M_1$. \par \vspace{2pt}

We forge\vspace{2pt} ahead with $M_2$. We assume for the rest of the proof that $p \neq 2$. We\vspace{0.25pt} take $\tau=\left(\begin{smallmatrix} 0 & 0 & -1 \\ 0 & 1 & 0 \\ 1 & 0 & 0 \end{smallmatrix}\right)$ in $N_G(T)$. As before,\vspace{-1.5pt} we must find those elements $v \in M_2$ such that the characteristic polynomial of $\tau v$ is $(1-z)^3$. This\vspace{0.5pt} only happens for $x:=\left(\begin{smallmatrix}
1 &  0 & 2 \\
0 &  1 & 0 \\
0 &  0 & 1
\end{smallmatrix}\right)$. The Jordan type\vspace{0.5pt} of $\tau x$ is $(2,1)$. Equation \ref{chiychi2fib} now takes the following form:
$$\begin{array}{l}
\chi(y_\chi^2) = {\displaystyle  \frac{\chi(C_{(2,1)})^2}{|M_2|} \sum_{v \in M_2} \chi\hspace{-1.5pt}\left(\tau x^2 v\tau v^{-1}\right)} \vspace{7pt} \\
\phantom{\chi(y_\chi^2)} = {\displaystyle \frac{(2q+1)^2(q-1)^4}{q^2} \sum_{a,b \in \F_q} \chi
\hspace{-1.5pt}\left(\hspace{-1pt}
\left(\begin{smallmatrix} 0 & 0  & -1 \\ 0 & 1 & 0 \\ 1 & 0 & 0 \end{smallmatrix}\right)
\hspace{-2.5pt}\left(\begin{smallmatrix} 1 & 0 & 2 \\ 0 &  1 & 0 \\ 0 &  0 & 1 \end{smallmatrix}\right)^{\hspace{-1pt} 2}
\hspace{-2.5pt}\left(\begin{smallmatrix} 1 & a & b \\  0 & 1 & a \\ 0 & 0 & 1 \end{smallmatrix}\right)
\hspace{-2.5pt}\left(\begin{smallmatrix} 0 & 0  & -1 \\ 0 & 1 & 0 \\ 1 & 0 & 0 \end{smallmatrix}\right)
\hspace{-2.5pt}\left(\begin{smallmatrix} 1 & \hspace{1pt}-a & a^2-b \vspace{-2.5pt} \\  0 & \hspace{1pt}1 & \hspace{4pt}-a\phantom{a^L} \vspace{-1.5pt} \\ 0 & \hspace{1pt}0 & 1 \end{smallmatrix}\hspace{-3pt}\right)\hspace{-1pt}
\right)} \vspace{5pt} \\
\phantom{\chi(y_\chi^2)} = {\displaystyle \frac{(2q+1)^2(q-1)^4}{q^2}\sum_{a,b \in \F_q} \chi \hspace{-1.5pt}\left(\begin{smallmatrix}
-1   & \hspace{6pt} a        & \hspace{6pt}b-a^2 \\
a    & \hspace{6pt}1-a^2     & \hspace{6pt}a(a^2-b-1) \\
b+4  & \hspace{6pt}-a(b+3)  & \hspace{6pt}(a^2-b)(b+4)-a^2-1
\end{smallmatrix}\right).}
\end{array}$$
We need to check next when the latter generic matrix belongs to $P$. A tedious (or computer) calculation shows that its characteristic polynomial is:
$$1+(1-2a^2+4b-a^2b+b^2)z+(-1+2a^2-4b+a^2b-b^2)z^2-z^3.$$
Setting this equal to $(1-z)^3$, we get the following equation:
$$4-2a^2+(4-a^2)b+b^2=0.$$
For each $a \in \F_q$ the values of $b$ satisfying this are $b=-2$ and $b=a^2-2$. \par \vspace{2pt}

For the pair $(0,-2)$, the corresponding matrix is of Jordan type $(2,1)$. For the pairs $(a,-2)$ and $(a,a^2-2)$, with $a \neq 0$, the corresponding matrices are of Jordan type $(3)$. Then:
$$\begin{array}{l}
\chi(y_\chi^2) = {\displaystyle \frac{(2q+1)^2(q-1)^4}{q^2} \Big((2q+1)(q-1)^2+2(q-1)(q-1)^2\Big)} \vspace{5pt} \\
\phantom{\chi(y_\chi^2)} = {\displaystyle \frac{(2q+1)^2(q-1)^6(4q-1)}{q^2}.}
\end{array}$$
From this, it follows that $\frac{1}{p} \in R$. This finishes the proof of the proposition. \par \vspace{2pt}
\end{proof}

We finally explain how to derive the conclusion of Theorem \ref{main1} for $\SL_3(p)$ from Proposition \ref{propSL3q}. The only missing link is the form of an arbitrary abelian $p$-subgroup of central type $M$ of $\SL_3(p)$. We will show that $M$ is isomorphic to one of the subgroups listed in Proposition \ref{propSL3q} through an automorphism of $\SL_3(p)$. Then, Remark \ref{twistiso} and Proposition \ref{propSL3q} will do the rest. \par \vspace{2pt}

By Sylow's theorem, $M$ is conjugate to a subgroup of $U$. We can assume that $M \subset U$ for our purpose. Observe that $U$ is now the Heisenberg group of order $p^3$. Set:
$$g_1=\left(\begin{smallmatrix} 1 & 1 & 0 \\ 0 &  1 & 0 \\ 0 &  0 & 1 \end{smallmatrix}\right),\hspace{2mm}
g_2=\left(\begin{smallmatrix} 1 & 0 & 1 \\ 0 &  1 & 0 \\ 0 &  0 & 1 \end{smallmatrix}\right), \hspace{2mm} \textrm{ and }\hspace{2mm}
g_3=\left(\begin{smallmatrix} 1 & 0 & 0 \\ 0 &  1 & 1 \\ 0 &  0 & 1 \end{smallmatrix}\right).$$
Recall that these matrices satisfy the following relations:
$$g_1^p=g_2^p=g_3^p=1,\hspace{6mm} g_1g_2=g_2g_1,\hspace{6mm} g_2g_3=g_3g_2,\hspace{6mm} g_3g_1=g_1g_3g_2^{p-1}.$$
Recall also that $Z(U)$ is generated by $g_2$. Since $M$ is abelian of central type, $\vert M \vert$ must be a square. Thus, $\vert M \vert=p^2$. The subgroup $M \cap Z(U)$ is trivial or $Z(U)$. If it were trivial, then we would have $U \simeq M \times Z(U)$, yielding that $U$ is abelian, a contradiction. Hence, $M \cap Z(U)=Z(U)$ and so $Z(U) \subset M$. Then, $M=\langle g_2, h \rangle$ for some $h \in M \menos Z(U)$ of order $p$. If $p \neq 2$, then every non-trivial element of $U$ has order $p$. If $p=2$, then $U \simeq D_4$ and all non-trivial elements have order $2$ except $g_3g_1$ and $g_2g_3g_1$. Therefore, the abelian $p$-subgroups of central type of $U$ are:
\begin{enumerate}
\item[(i)] Case $p \neq 2$: $L_j:=\langle g_2, g_3g_1^j \rangle,$ with $j \in \{0,\ldots, p-1\},$ and $L_p:=\langle g_2, g_1 \rangle$. \vspace{2pt}
\item[(ii)] Case $p=2$: $L_0:=\langle g_2, g_3\rangle$  and $L_2:=\langle g_2, g_1 \rangle$.
\end{enumerate}

These subgroups are related via automorphisms as follows. Let $J$ be the monomial matrix with $1$'s on the antidiagonal. Consider the automorphism $$\Psi:\SL_3(p) \mapsto \SL_3(p),\ A \mapsto J{}^{\hspace{1.5pt} t}\hspace{-1pt}(A^{-1})\hspace{1pt}J^{-1}.$$ One can easily verify that $\Psi(L_p)=L_0$. Suppose now that $p \neq 2$ and take $i,j \in \{1,\ldots, p-1\}$. Set $k=\textrm{diag}(i,j,j)$. One can easily see as well that conjugation by $k$ induces an automorphism in $\SL_3(p)$ that maps $L_j$ into $L_i$. \par \vspace{2pt}

Getting back to the notation of Proposition \ref{propSL3q}, it is $M_1=L_p$ and $M_2=L_1$, when $p \neq 2$, and $M_1=L_2$, when $p=2$. The previous discussion shows that any abelian $p$-subgroup of central type of $\SL_3(p)$ can be carried to $M_1$ or $M_2$ through an automorphism of $\SL_3(p)$ when $p \neq 2$ and to $M_1$ when $p=2$. \par \vspace{2pt}

This finishes the proof of Theorem \ref{main1} for $\SL_3(p)$. \qed

\begin{remark}
There is no automorphism of $\SL_3(p)$ that maps $L_0$ into $L_j$ for $j \neq 0, p$. The description of the automorphisms of a finite group of Lie type, in particular of $\SL_3(p)$, can be found in \cite[Theorem 24.24 and page 215]{MT}. Such a description unveils that every automorphism preserves the Jordan canonical form of $p$-elements. The non-trivial elements in $L_0$ are of Jordan type $(2,1)$, whereas $g_3g_1^j$ in $L_j$ is of Jordan type $(3)$ for $j \neq 0$.
\end{remark}

The determination of the abelian $p$-subgroups of central type of $\SL_3(q)$ needs extra work since the additive group of $\F_q$ contributes with new subgroups. We will not delve into this in order to not deviate much from the course set by Theorem \ref{Thompson}.

\subsection{The Suzuki groups}
We recall the construction of the Suzuki group $^2\!B_2(q)$ from \cite[Section 13]{Su1}. In this case, $q=2^{2n+1}$ with $n \geq 1$. Consider the Frobenius automorphism $\Fr_2$ of $\F_q$. Set $\theta = \Fr_2^{n+1}$. Then, $\theta^2=\Fr_2$ and the fixed field $(\F_q)^{\theta}$ equals $\F_2$. The following remark is useful to work with Equation \ref{eq:tU} below:

\begin{remark}\label{rem:condizioni} Observe that:
\begin{enumerate}
\item If $a \in \F_q$ satisfies $a\theta(a)=1$, then $1=\theta(a)a^2$. Hence, $a=a^2$ and, consequently, $a \in \F_2$. \vspace{2pt}

\item The map $\Phi:\Fq^\times \to \Fq^\times, a \mapsto a\theta(a),$ is an isomorphism.
\end{enumerate}
\end{remark}

The group $^2\!B_2(q)$ was defined in \cite{Su1} as a subgroup of $\SL_4(\F_q)$ generated by matrices of a certain type. The first family of such matrices is the following:
$$u(a,b):=
\left(\begin{matrix}
1   & \hspace{-2pt} 0 & \hspace{9pt} 0 & \hspace{11pt} 0 \\
a & \hspace{-2pt} 1 & \hspace{9pt} 0 & \hspace{11pt} 0 \\
a\theta(a)+b & \hspace{-2pt} \theta(a) & \hspace{9pt} 1 & \hspace{11pt} 0 \\
a^2\theta(a)+ab+\theta(b) & \hspace{-2pt} b & \hspace{9pt} a & \hspace{11pt} 1
\end{matrix}\right), \hspace{6mm} a,b\in \F_q.$$
These matrices satisfy the multiplication rule:
$$u(a,b)u(a',b')=u(a+a',a\theta(a')+b+b').$$
Let $U$ denote the subgroup of $\SL_4(\F_q)$ that they generate. We have that $\vert U \vert=q^2$. Moreover, $U$ enjoys the following properties, see \cite[Lemma 1]{Su1}:
\begin{enumerate}
\item $U$ has exponent $4$. \vspace{2pt}
\item The center of $U$ is generated by the elements $u(0,b)$, with $b \in \Fq$. It is an elementary abelian $2$-group of order $q$. \vspace{2pt}
\item An element of $U$ is an involution if and only if it belongs to $Z(U)$.
\end{enumerate}
\vspace{2pt}

The second family of matrices is parameterized by $\kappa \in \F_q^{\times}$. They are:
$$t_{\kappa}:=\textrm{diag}(a_1,a_2,a_2^{-1},a_1^{-1}),\hspace{2mm} \textrm{ with }\hspace{3pt} \theta(a_1)=\kappa\theta(\kappa) \hspace{3pt} \textrm{ and } \hspace{3pt} \theta(a_2)=\kappa.$$
They form a subgroup $T$ of $\SL_4(\F_q)$, which is isomorphic to $\F_q^{\times}$. The following formula holds:
\begin{equation}\label{eq:tU}
t_{\kappa}^{-1}u(a,b)t_{\kappa}=u(a\kappa,b\kappa\theta(\kappa)).
\end{equation}
In particular, $T$ normalizes $U$. The subgroup $TU$ is isomorphic to the semidirect product $T \ltimes U$. Finally, let $\tau$ be the monomial matrix with $1$'s on the antidiagonal. The Suzuki group $^2\!B_2(q)$ is the subgroup of $\SL_4(\F_q)$ generated by $U,T,$ and $\tau$. We have (\cite[Theorem 7]{Su1}):
$$\left\vert {}^2\!B_2(q)\right\vert=q^2(q-1)(q^2+1).$$

The subgroup $U$ is a Sylow $2$-subgroup of $^2\!B_2(q)$ by \cite[Theorem 7]{Su1}. Therefore, all non-trivial involutions are conjugate to some $u(0,b)$. By Remark \ref{rem:condizioni}(2) and \eqref{eq:tU}, all non-trivial involutions are conjugate to $u(0,1)$. By \cite[Propositions 1, 2, and 3]{Su1}, the centralizer of a non-trivial element $u(a,b)$ is contained in $U$. \par \vspace{2pt}

It can be checked that every element of $^2\!B_2(q)$ leaves invariant the bilinear form defined by $\tau$. This permits to regard $^2\!B_2(q)$ as a subgroup of the symplectic group $\Sp_4(\F_q)$. The description of the Steinberg endomorphism giving rise to $^2\!B_2(q)$ can be found in \cite{O1,O2} and \cite[Section 12.3]{Ca}. \par \vspace{2pt}

In this subsection, the data for our setting are: $G={^2}\!B_2(q)$, and $U$ and $\tau$ as above. The following lemma and Remark \ref{twistiso} will allow us to take $M$ as a subgroup of $Z(U)$:

\begin{lemma}
Let $M$ be an abelian $2$-subgroup of $^2\!B_2(q)$ of central type. Then, $M$ is conjugate to a subgroup of $Z(U)$. In particular, $M$ is generated by involutions.
\end{lemma}

\begin{proof}
By Sylow's theorem, $M$ is conjugate to a subgroup of $U$. Bear in mind that $U$ has exponent $4$. Since $M$ is of central type, $M \simeq E \times E$ for some group $E$. We next see that $E$ does not have an element of order $4$. If it were so, $U$ would contain a subgroup isomorphic to $C_4 \times C_4$. Pick $u(a,b)$ and $u(a',b')$ two generators of such a subgroup. They have order $4$, so necessarily $aa' \neq 0$. Furthermore, they commute. This means:
$$u(a'+a,a'\theta(a)+b'+b)=u(a+a',a\theta(a')+b+b').$$
This gives $a'\theta(a) = a\theta(a')$. Then, $a'a^{-1} = \theta(a'a^{-1})$. Thus, $a'a^{-1} \in \F_2$, which implies $a'=a$. But, then $a+a'=0$ and the element $u(a,b)u(a',b')=u(0,a\theta(a)+b+b')$ has order $2$ and not $4$, a contradiction. Therefore, $M$ has exponent $2$. By the properties of $U$ recalled before, $M$ is conjugate to a subgroup of $Z(U)$.
\end{proof}

We consider the induced character $\chi=\Ind_{U}^{G}(\Uno_{U})$ and the set $P=\bigcup_{g\in G} g\hspace{0.25pt} U\hspace{-1.25pt} g^{-1}$. We know that $\chi$ vanishes outside $P$. We compute its values at $U$. We strongly rely on the properties of $U$ and $T$ listed before. We start with a non-trivial involution $u(0,b)$. It is conjugate to $u(0,1)$ by a certain $t_{\kappa} \in T$ in view of \eqref{eq:tU}. Note that $gu(0,1)g^{-1}\in U$ implies $gu(0,1)g^{-1}\in Z(U)$. Using \eqref{eq:tU} and that $C_{G}(u(0,1))=U$, one can check the equality $\{g\in G : g u(0,1) g^{-1} \in U\} = TU.$ As a consequence of this, and \eqref{eq:tU} again, the same equality holds for $u(0,b)$. Now, an arbitrary $u(a,b)$, which is not an involution, has order $4$. Then, we have the following chain of inclusions:
$$TU \subseteq \{g\in G : g u(a,b) g^{-1} \in U\}  \subseteq \{g\in G : g u(a,b)^2 g^{-1} \in U\} = TU.$$
Hence:
$$\begin{array}{l}
\chi(1) = {\displaystyle \frac{|G|}{|U|} = (q-1)(q^2+1),} \vspace{7pt} \\
\chi(v) = {\displaystyle \frac{|T U|}{|U|}=q-1}, \hspace{5mm} \textrm{ for } v \in U \menos \{1\}.
\end{array}$$

In this case, the set $I_{\chi}$ particularizes to $I_\chi=\{\{1\},P\menos\{1\}\}$. Let $P^{\bullet}=P\menos\{1\}$. Clearly, $M_{\{1\}}=\emptyset$. The next step is to find $M_{P^{\bullet}}$. We analyze when the following product is a $2$-element:
\begin{equation*}
\tau u(0,b) =\left(\begin{smallmatrix}
0\phantom{b} & \hspace{-2pt} 0 & \hspace{2pt} 0 & \hspace{-5.1mm} 1\\
0\phantom{b} & \hspace{-2pt} 0 & \hspace{2pt} 1 & \hspace{-5.1mm} 0\\
0\phantom{b} & \hspace{-2pt} 1 & \hspace{2pt} 0 & \hspace{-5.1mm} 0\\
1\phantom{b} & \hspace{-2pt} 0 & \hspace{2pt} 0 & \hspace{2pt} 0\phantom{\theta(b)}
\end{smallmatrix}\hspace{-5.75mm}\right) \hspace{-1pt}
\left(\begin{smallmatrix}
   1 & 0 & \hspace{3pt} 0 & \hspace{3pt} 0 \phantom{b}\\
   0 & 1 & \hspace{3pt} 0 & \hspace{3pt} 0 \phantom{b}\\
b & 0 & \hspace{3pt} 1 & \hspace{3pt} 0\phantom{b} \\
\theta(b) & b & \hspace{3pt} 0 & \hspace{3pt} 1 \phantom{b}
\end{smallmatrix}\hspace{-3pt} \right)
=
\left(\begin{smallmatrix}
\theta(b) & b & \hspace{3pt} 0 & \hspace{3pt} 1 \phantom{b}\\
b & 0 & \hspace{3pt} 1 & \hspace{3pt} 0 \phantom{b}\\
0 & 1 & \hspace{3pt} 0 & \hspace{3pt} 0 \phantom{b} \\
1 & 0 & \hspace{3pt} 0 & \hspace{3pt} 0 \phantom{b}
\end{smallmatrix}\hspace{-3pt} \right).
\end{equation*}
One can see, by looking at the $(4,3)$-entry, that $(\tau u(0,b))^4=1$ if and only if $b=0$. Then:
$$M_{P^{\bullet}}=\{v\in M : \tau v \in P^{\bullet}\}=\{1\}.$$

Equation \ref{chiychi2fib} now takes the following form. Notice that, in view of the preceding discussion, $(\tau v)^2 \in P$ if and only if $v=1$. Hence, the only contribution in the second sum occurs when $v=1$:
$$\begin{array}{l}
\chi(y_\chi^2) = {\displaystyle \frac{\chi(P^{\bullet})^2}{|M|} \sum_{\begin{subarray}{c} v \in M \vspace{0.8pt} \\ x, x' \in M_{\hspace{-0.75pt}P^{\bullet}} \end{subarray}} \chi \big(\tau x x'v \tau v^{-1}\big)} \vspace{7pt} \\
\phantom{\chi(y_\chi^2)} = {\displaystyle \frac{\chi(P^{\bullet})^2}{|M|} \sum_{v \in M} \chi \big((\tau v)^2\hspace{1pt}\big)} \vspace{7pt} \\
\phantom{\chi(y_\chi^2)} = {\displaystyle \frac{(q-1)^3(q^2+1)}{|M|}.}
\end{array}$$
This is an irreducible fraction because $|M|$ is a power of $2$. As in the previous cases, we can derive from this that $\frac{1}{2} \in R$. \par \vspace{2pt}

This finishes the proof of Theorem \ref{main1} for $^2\!B_2(q)$. \qed \vspace{4pt}

We are just one step away from a complete analysis of $^2\!B_2(q)$. The next result shows that there are no other abelian subgroups of central type:

\begin{proposition}\label{Suz:classct}
Every non-trivial abelian subgroup of central type $M$ of $^2\!B_2(q)$ is a $2$-group.
\end{proposition}

\begin{proof}
This follows from the knowledge of the structure of the Sylow subgroups of $^2\!B_2(q)$. By \cite[Theorem 3.9, page 189]{HB}, for $p$ odd, every Sylow $p$-subgroup of $^2\!B_2(q)$ is cyclic. Now, write $M \simeq E \times E$ for some subgroup $E$. Let $r$ be a prime divisor of $\vert E \vert$. Two elements of order $r$, one in each copy of $E$, generate a non-cyclic subgroup of order $r^2$. This is not possible if $r$ is odd, as such a subgroup must be contained in a Sylow $r$-subgroup of $^2\!B_2(q)$, that is cyclic.
\end{proof}

Proposition \ref{Suz:classct} rounds off Corollary \ref{complexif} in the following way:

\begin{corollary}
Let $\Omega$ be a non-trivial twist of $\Co\hspace{1pt} {}^2\!B_2(q)$ arising from a $2$-cocycle on an abelian subgroup of\hspace{1.5pt} ${}^2\!B_2(q)$. Then, the complex semisimple Hopf algebra $(\Co\hspace{1pt} {}^2\!B_2(q))_{\Omega}$ does not admit a Hopf order over any number ring.
\end{corollary}

This provides us with the first family of simple groups for which the non-existence of integral Hopf orders is established for \emph{any} non-trivial twist arising from an abelian subgroup. \par \vspace{2pt}

\section{Twists of $\PSL_2(q)$ arising from the Klein four-group}

In the families of simple groups studied before, there is, by construction, a natural Sylow $p$-subgroup available whose structure is fully understood. This fact made possible our description, up to automorphisms, of the abelian $p$-subgroups of central type. However, the determination of other kind of abelian subgroups of central type and their interrelations can be a difficult task. \par \vspace{2pt}

In this section we will find, up to conjugation or automorphisms, \emph{all} abelian subgroups of central type of $\PSL_2(q)$ and prove that the corresponding twisted group algebra does not admit an integral Hopf order.

\subsection{Abelian subgroups of central type in $\PSL_2(q)$}
\enlargethispage{\baselineskip}

The classification of all subgroups of $\PSL_2(q)$, which dates from the beginning of last century, can be consulted in \cite[Theorem 6.25]{Su2} or \cite[Theorem 2.1]{Ki}. These theorems bring to light that the only abelian subgroups of central type of $\PSL_2(q)$ are: $p$-groups of square order or Klein four-groups. In this subsection we give a self-contained proof of this and we find the relation between these subgroups via conjugation or automorphisms. \par \vspace{2pt}

We start by recalling some facts about $\SL_2(q)$ that we will draw heavily on in the sequel. We used as references \cite[pages 5-9]{B}, for $q$ odd, and \cite[pages 324-326 and 336]{JL}, for $q$ even. \par \vspace{2pt}

We said before that
$$U:=\left\{\left(\begin{smallmatrix} 1 & a \\ 0 & 1\end{smallmatrix}\right) : a \in \F_q \right\}$$
is a Sylow $p$-subgroup of $\SL_2(q)$, which is elementary abelian. For $t \in \F_q^{\times}$ we denote by $d(t)$ the diagonal matrix $\textrm{diag}(t,t^{-1})$. We fix the following split torus of $\SL_2(q)$:
$$T:=\left\{d(t):t \in \F_q^{\times}\right\}.$$
We next fix a non-split torus $T'$ of $\SL_2(q)$. It is constructed by realizing the elements of norm $1$ of the field extension $\F_{q^2}/\F_q$ as matrices of size $2$ in the following way: \par \vspace{2pt}

Case $q$ odd. Let $\epsilon \in \F_q$ be a non-square element. Take $\zeta \in \F_{q^2}$ such\vspace{0.5pt} that $\zeta^2=\epsilon$. Every\vspace{0.5pt} element of $\F_{q^2}$ is of the form $a+b\zeta$, with $a,b \in \F_q$. The following map is an algebra morphism:
$$d':\F_{q^2} \rightarrow \textrm{M}_2(\F_q),\ a+b\zeta \mapsto \left(\begin{smallmatrix} a & b \\ \epsilon b & a\end{smallmatrix}\right).$$
We set:
$$T'=\left\{d'\hspace{-0.5pt}(a+b\zeta\hspace{0.3pt}): (a+b\epsilon)(a-b\epsilon)=1\right\}.\vspace{5pt}$$

Case $q$ even. Take now $\zeta \in \F_{q^2} \menos \F_q$. Then, $\zeta+\zeta^q$ and $\zeta^{1+q}$ belong to $\F_q$ and the following map is an algebra morphism:
$$d':\F_{q^2} \rightarrow \textrm{M}_2(\F_q),\ a+b\zeta \mapsto \left(\begin{smallmatrix} a & \hspace{2pt} b \vspace{2pt} \\ b\zeta^{1+q} & \hspace{2pt} a+b(\zeta+\zeta^q) \end{smallmatrix}\right).$$
We set:
$$T':=\left\{d'\hspace{-0.5pt}(a+b\zeta\hspace{0.3pt}): (a+b\zeta)(a+b\zeta^q)=1\right\}.\vspace{5pt}$$

We will use the following facts in the proof of the next proposition:
\begin{enumerate}
\item The split and non-split torus $T$ and $T'$ are cyclic groups of orders $q-1$ and $q+1$ respectively. \vspace{2pt}
\item Every non-central element of $\SL_2(q)$ is conjugate to either $\pm u,$ with $u \in U$, an element in $T$ or in $T'$. \vspace{2pt}
\item Let $g \in \SL_2(q)$ be such that $g \neq \pm 1$. Then, $C(g)=\{\pm 1\}U$ if $g \in U$, $C(g)=T$ if $g \in T$, and $C(g)=T'$ if $g \in T'$. \vspace{2pt}
\item Every element of order coprime with $p$ (i.e., semisimple) is conjugate to an element of $T \cup T'$.
\end{enumerate}
\enlargethispage{2\baselineskip}

We keep the notation $\pi:\SL_2(q) \rightarrow \PSL_2(q)$ for the canonical projection. \par \vspace{2pt}

To describe the action of $\textrm{Aut}(\PSL_2(q))$ on the set of subgroups of central type of $\PSL_2(q)$, we view $\PSL_2(q)$ inside $\PGL_2(q)$ as the image of $\SL_2(q)$ through the canonical projection $\GL_2(q) \twoheadrightarrow \PGL_2(q)$. Under this identification, $\PSL_2(q)=\PGL_2(q)$ if $q$ is even, and, if $q$ is odd, $\PSL_2(q)$ is the unique proper normal subgroup of $\PGL_2(q)$, see \cite[Section 1]{Su3}. The action by conjugation of $\PGL_2(q)$ on $\PSL_2(q)$ gives rise to an injective group homomorphism $\PGL_2(q) \rightarrow \textrm{Aut}(\PSL_2(q))$.

\begin{proposition}\label{subPSL2}
Let $M$ be a non-trivial abelian subgroup of central type of $\PSL_2(q)$. Then, $M$ is one of the following subgroups:
\begin{itemize}
\item[(i)] A subgroup of $\pi(U)$ of square order, up to conjugation. \vspace{2pt}
\item[(ii)] A subgroup isomorphic to $C_2\times C_2$. Moreover, when $q$ is odd, there is a single orbit for the action of $\PGL_2(q)$ (and thus of $\Aut(\PSL_2(q))$) on the set consisting of such subgroups.
\end{itemize}
\end{proposition}

\begin{proof}
The proof is divided into two parts. In the first part we describe how an abelian subgroup of central type is. In the second part we deal with the statement about the orbit for the action on the Klein four-groups. \par \vspace{2pt}

\emph{1. Description of $M$.} We first show that $M$ must be either a $p$-group or a $2$-group. The justification of this assertion will cover item (i). \par \vspace{2pt}

Since $M$ is abelian of central type, we know that $M \simeq E \times E$ for some subgroup $E$. Let $r$ be a prime divisor of $\vert E \vert$. We will see that $r=2$ or $r=p$. Let $\bar{g},\bar{h}$ be elements of $M$, one in each copy of $E$, such that $\ord(\bar{g})=\ord(\bar{h})=r$. Then, $\bar{g}$ and $\bar{h}$ commute and $\bar{h} \notin \langle \bar{g} \rangle$. {Put $\bar{g}={\pi}(g)$ and $\bar{h}={\pi}(h)$ for some $g,h \in \SL_2(q)$}. 

We distinguish two cases: \par \vspace{4pt}

I. Case $q$ even. Every element of $\SL_2(q)$ is conjugate to an element of $U,T$ or $T'$. Suppose that $g$ were conjugate to an element of $T$. We would have that $C(g)=T$ and that $g$ and $h$ generate a non-cyclic subgroup of $C(g)$. This is not possible because $T$ is cyclic. The same argument applies if $g$ were conjugate to an element of $T'$. Assume that $g$ is conjugate to an element $u$ of $U$. Then, $M$ is conjugate to a subgroup of $C(u)$. The latter equals $U$. Hence, $r=2$ and $M$ is conjugate to a $2$-group of square order. This establishes item (i) for $q$ even. \par \vspace{4pt}

II. Case $q$ odd. Every element of $\SL_2(q)$ is conjugate\vspace{0.25pt} to an element of the following form: $\pm u$, with $u \in U$, $d(t)$ or $d'(a+b\zeta\hspace{0.5pt})$. The elements $g$ and $h^2$ commute and $g$ is non-central. Proceed as before with these two elements and the subgroup generated by them. Take into account that $C(u)$ is now $\{\pm 1\}U$ for $u \in U$ with $u \neq 1$. We obtain that $h^2=1$ or $h^{4p}=1$. Hence, $r=2$ or $r=p$. \par \vspace{2pt}

Suppose that $p$ divides $\vert E \vert$. We take $\bar{g}$ and $\bar{h}$ of order $p$. Then, $g$ is conjugate to $u$ or $-u$, for some $u \in U$, and $M$ is conjugate to a subgroup of ${\pi}(C(u))$. As the latter equals ${\pi}(U)$, item (i) follows for $q$ odd. On the other hand, if $p$ does not divide $\vert E \vert$, then, according to the previous paragraph, the only prime divisor of $\vert E \vert$ is $2$, and every non-trivial element of $M$ has order $2$. That is, $M$ is an elementary abelian $2$-group. \par \vspace{2pt}

For the rest of the proof, we assume that $p$ is odd and $M$ is an elementary abelian $2$-group. We know that every element of $\PSL_2(q)$ of order $2$ is conjugate to an element in ${\pi}(T \cup T')$. By reason of orders, ${\pi}(T \cup T')$ has a unique element of order $2$. Up to conjugation, we can assume that $M$ contains such an element which we denote again by $\bar{h}$. We now distinguish two cases for $q$: \par \vspace{3pt}

A. Case $q\equiv_4 1$. In this case, $\F_q$ has\vspace{0.5pt} a primitive fourth root of unity, say $\eta$. Suppose that $\bar{h}={\pi}(d(\eta))={\pi}\hspace{-2pt}\left(\begin{smallmatrix} \eta & 0 \\ 0 & \eta^{-1} \end{smallmatrix}\right)$. One can check that the centralizer of $\bar{h}$ is ${\pi}(T \cup \rho\vspace{0.5pt} T)$, where
$\rho = \left(\begin{smallmatrix} 0 & 1 \\ -1 & 0 \end{smallmatrix}\right)$. Then,\vspace{1pt} $\bar{g}={\pi}\hspace{-1.5pt}\left(\begin{smallmatrix} 0 & \alpha \\ -\alpha^{-1} & 0 \end{smallmatrix}\right)$, for some $\alpha \in \F_q^\times$. We  have that $M$ is contained in the centralizer \vspace{2pt} $C(\langle \bar{g},\bar{h} \rangle)$ of $\langle \bar{g},\bar{h} \rangle$. A direct calculation shows \vspace{2pt} the equality $C(\langle \bar{g},\bar{h} \rangle)=\langle \bar{g},\bar{h} \rangle$, so $M=\langle \bar{g},\bar{h} \rangle$ and $M \simeq C_2 \times C_2$. \par \vspace{3pt}

B. Case $q\equiv_4 -1$. In this case, $-1$ is not a square in $\F_q$. We define\vspace{1pt} the non-split torus $T'$ by $\zeta \in \F_{q^2}$ such that $\zeta^2=-1$. The only element\vspace{1pt} of order $2$ in ${\pi}(T \cup T')$ is now $\bar{h}={\pi}(d'(\zeta))= {\pi}\hspace{-1.5pt}\left(\begin{smallmatrix} 0 & 1 \\ -1 & 0 \end{smallmatrix}\right)$. Consider the following subset of $\SL_2(q)$:\vspace{-1.5mm}
$$S:=\left\{\left(\begin{smallmatrix} x & y \\ y & -x \end{smallmatrix}\right):\hspace{2pt} x,y\hspace{1pt}  \in \F_q \textrm{ and } \hspace{1pt} x^2+y^2=-1\right\}.\vspace{-1.5mm}$$
Notice that $\vert S \vert = q+1$. One can\vspace{2pt} verify that the centralizer of $\bar{h}$ equals ${\pi}(T'\cup S)$. Then, $\bar{g}={\pi}\hspace{-1.5pt}\left(\begin{smallmatrix} x & y \\ y & -x \end{smallmatrix}\right)$, for some $x,y \in \F_q$ as above. The\vspace{1pt} equality $C(\langle \bar{g},\bar{h} \rangle)=\langle \bar{g},\bar{h} \rangle$ holds in this case as well. It implies that $M=\langle \bar{g},\bar{h} \rangle$ and $M \simeq C_2 \times C_2$. \par \vspace{6pt}

\emph{2. Transitivity of the actions on the Klein four-groups.} Let $\Mcal$ denote the set consisting of subgroups of $\PSL_2(q)$ that are isomorphic to $C_2 \times C_2$. Bear in mind our identification of $\PSL_2(q)$ with a normal subgroup of $\PGL_2(q)$. We will prove that the action by conjugation of  $\PGL_2(q)$ on $\Mcal$ is transitive. This will imply that the action of $\textrm{{Aut}}(\PSL_2(q))$ is as well. It will suffice to check that the following inequality holds for $M \in \Mcal$:
$$|\Mcal| \leq \frac{|\PGL_2(q)|}{|\textrm{Stab}(M)|}=\frac{|\PGL_2(q)|}{|\textrm{N}_{\PGL_2(q)}(M)|}.$$ (This will be actually an equality as there will be only one orbit for the action). The cardinality of $\Mcal$ is $\frac{|\PSL_2(q)|}{12}$, see \cite[Exercise 5(c), page 417]{Su2}.  With the information available in this proof, this can be deduced from the fact that the centralizer of $\bar{h}$ is isomorphic to a dihedral group of order $q-1$ when $q\equiv_4 1$ and of order $q+1$ when $q\equiv_4 -1$. Hence, we will need to verify that
$$\frac{q(q^2-1)}{24} |\textrm{N}_{\PGL_2(q)}(M)| \leq q(q^2-1),$$
i.e., that $|\textrm{N}_{\PGL_2(q)}(M)| \leq 24$. This inequality\vspace{0.5pt} can be attained in the following way. The action of $\textrm{N}_{\PGL_2(q)}(M)$ on the set $\{\bar{g},\bar{h},\bar{g}\bar{h}\}$ of non-trivial elements\vspace{0.5pt} in $M$ induces a group homomorphism $\textrm{N}_{\PGL_2(q)}(M) \to \mathbb{S}_3$ whose kernel\vspace{0.5pt} is the centralizer $C_{\PGL_2(q)}(M)$ of $M$ in $\PGL_2(q)$. A direct calculation\vspace{0.5pt} shows that $C_{\PGL_2(q)}(M)=M$. Therefore, $|\textrm{N}_{\PGL_2(q)}(M)| \leq 4|\mathbb{S}_3|=24$. This establishes\vspace{0.5pt} the second part of the statement of (ii) and finishes the proof.
\end{proof}

\begin{remark}
The number of conjugacy classes in $\PSL_2(q)$ of the Klein four-groups is one if $q\equiv_8 \pm 3$ and two if $q\equiv_8 \pm 1$; see \cite[Theorem 2.1, items (j) and (k)]{Ki} or \cite[Exercise 5(d), page 417]{Su2}.
\end{remark}

\subsection{Non-existence of integral Hopf orders}

The goal of this subsection is to establish the following result:

\begin{theorem}\label{main2}
Let $K$ be a number field and $R\subset K$ a Dedekind domain such that $\Oint_K \subseteq R$. Let $p$ be\vspace{1pt} a prime number and $q=p^m$ with $m \geq 1$. (We assume that $m>1$ when $p=2,3$.) Let $M$ be a non-trivial subgroup\vspace{-0.25pt} of central type of $\PSL_2(q)$ and $\omega: \widehat{M} \times \widehat{M} \rightarrow K^{\times}$ a normalized non-degenerate cocycle. \par \vspace{2pt}

If the twisted Hopf algebra $(K \PSL_2(q))_{\Omega_{M,\omega}}$ admits a Hopf order over $R$, then $\frac{1}{\vert M \vert} \in R$. Hence, $(K \PSL_2(q))_{\Omega_{M,\omega}}$ does not admit a Hopf order over $\Oint_K$.
\end{theorem}

We will use several characters of $\PSL_2(q)$ in the proof. Composition with the projection $\pi:\SL_2(q) \to \PSL_2(q)$ induces a one-to-one correspondence between characters of $\PSL_2(q)$ and characters of $\SL_2(q)$ with kernel containing $\{\pm 1\}$; see \cite[Lemma 2.22]{I}. In particular, if $\varphi$ is a character of $\SL_2(q)$ such that $\varphi(1)=\varphi(-1)$, then $\tilde{\varphi}:\PSL_2(q) \rightarrow \Co, \pi(g) \mapsto \varphi(g)$ is a character of $\PSL_2(q)$. \par \vspace{2pt}

For convenience, we will work directly with the character table of $\SL_2(q)$. We will use that in \cite[Table 5.4, page 58]{B}. We entirely adopt the notation fixed there (we overlook the clash with our $R$, which is easily resolved from the context). The necessary information to work with this table can be found in the following parts of \cite{B}. For the conjugacy classes, see: Equation 1.1.9 in page 5, Theorem 1.3.3 in page 8, Table 1.1 in page 9, and Exercise 1.4(d) in page 12. For the irreducible characters, see: Subsection 3.2.3 in page 32, Summary 3.2.5 in page 34, Remark in page 35, Section 4.3 in page 45, Exercise 4.1(c) in page 48, and Proposition 5.3.1 in page 57. \par \vspace{2pt}

The original description of the character table of $\SL_2(q)$ given by Schur can be found in \cite[Theorem 38.1]{D}. The character table of $\PSL_2(q)$ appears, for instance, in \cite[Theorems 8.9 and 8.11, pages 280-282]{Ka2}. We stress that its size is $\frac{q+5}{2}$ when $q$ is odd.
\par \smallskip

{\it Proof of Theorem \ref{main2}}. In view of Proposition \ref{subPSL2} and Theorem \ref{main1} we only have to prove the statement when $p$ is odd and $M$ is isomorphic to the Klein four-group. Fix $(x,y) \in \F_p \times \F_p$ such that $x^2+y^2=-1$. By Remark \ref{twistiso} and Proposition \ref{subPSL2}(ii) we can assume furthermore that $M$ is the subgroup of $\PSL_2(q)$ generated by
$r=\pi \hspace{-1.5pt} \left(\begin{smallmatrix}
0 & 1 \\
-1 & 0
\end{smallmatrix}\right)$ and $s=\pi \hspace{-1.5pt} \left(\begin{smallmatrix}
x & y \\
y & -x
\end{smallmatrix}
\right)$. We will use two different characters in the proof according to the following distinction: \par \vspace{3pt}

A. Case $q \equiv_4 -1$. This case follows the same strategy as that of Subsection \ref{PSL(2,q)}. We take the induced character $\chi=\Ind_{\pi(U)}^{\PSL_2(q)}(\Uno_{\pi(U)})$ from the Sylow $p$-subgroup $\pi(U)$ of $\PSL_2(q)$. We know that its non-zero values are: \vspace{-3pt}
$$\chi(1)= \frac{q^2-1}{2} \hspace{6.5mm} \textrm{and} \hspace{6.5mm} \chi \hspace{-1pt}\left( \pi \hspace{-2pt}\left(\begin{smallmatrix} 1 & c \\
0 & 1\end{smallmatrix}\right)\right) = \frac{q-1}{2}, \hspace{3mm} \textrm{for }c \neq 0.\vspace{-3pt}$$

B. Case $q \equiv_4 1$. Set $\Theta=\{\theta \in  \overline{\mu_{q+1}}: \theta(1)\hspace{1pt} =\hspace{1pt} \theta(-1)\}$. Here, $\overline{\mu_{q+1}}$ denotes the set parameterizing the family of irreducible characters $R'(\theta)$ of $\SL_2(q)$, see \cite[page 45]{B}. We consider the character \vspace{-3pt}
$$\varphi = \Uno_G - \textrm{St} - 2 \sum_{\theta \in \Theta} R'(\theta).\vspace{-3pt}$$
Observe that $\varphi(1)=\varphi(-1)$, so we can view $\varphi$ as a character of $\PSL_2(q)$. Its values are given in the following table: \vspace{1.5mm}

\begin{center}
\scalebox{0.93}{$\begin{array}{|c|ccccc|} \hline
\textrm{Conjugacy class}       & \pi(1)           & \hspace{-2.1mm}\pi(u_+)        & \hspace{1mm} \pi(u_{-})    & \hspace{1mm} \pi(d(a)) & \hspace{1mm} \pi(d'(\xi)) \\ \hline
\varphi & \hspace{8pt}\frac{1-q^2}{2}^{\phantom{M}}_{\phantom{p}} & \hspace{-2.1mm} \frac{q+1}{2} & \hspace{1mm} \frac{q+1}{2} & \hspace{1mm} 0    &   \hspace{1mm} 0 \\ \hline
\end{array}$}\vspace{1.5mm}
\end{center}
The elements $\pi(u_+)$ and $\pi(u_{-})$ have order $p$, whereas the elements of the form $\pi(d(a))$ and $\pi(d'(\xi))$ do not. \par \vspace{2pt}

As in Subsection \ref{PSL(2,q)}, we write $P$ for the set of elements of order $p$ together with the identity element. Recall\vspace{0.5pt} that $P=\{\pi(A) \in \PSL_2(q) : \Tr(A)=\pm 2\}$. Define $\psi=\varphi$ if $q \equiv_4 1$ and $\psi=\chi$ if $q \equiv_4 -1$. Observe\vspace{1.5pt} that $\psi$ vanishes outside $P$. Define also $n_q=\frac{q+1}{2}$ if $q \equiv_4 1$ and $n_q=\frac{q-1}{2}$ if $q \equiv_4 -1$. \par \vspace{3pt}

We need to treat separately the first three values of $p$: \par \vspace{3pt}
\enlargethispage{-3.5\baselineskip}

1. Case $p=3$. Here we assume that $m>1$ since $\PSL_2(3)$ is not simple.\vspace{1pt} We consider the pair $(x,y)=(1,1)$.
Take $\tau=\pi \hspace{-1.5pt} \left(\begin{smallmatrix}
1 & 0 \\
\lambda & 1
\end{smallmatrix}\right)$, where $\lambda \in \F_q \menos \F_3$. One can check\vspace{1pt} that this choice of $\lambda$ ensures that $M \cap (\tau M \tau^{-1})=\{1\}$. We compute the element $y_{\psi}$ as in \eqref{ychi}:

\begin{equation}\label{computation1}
\begin{array}{l}
y_{\psi} = {\displaystyle \frac{1}{\vert M \vert} \sum_{v,v' \in M} \psi(v\tau v')v\tau v'} \vspace{2pt} \\
\phantom{y_{\psi}} = {\displaystyle \frac{1}{\vert M \vert} \sum_{v,v' \in M} \psi(\tau v'v)v \tau v'} \vspace{2pt} \\
\phantom{y_{\psi}} = {\displaystyle \frac{1}{\vert M \vert} \sum_{v,v' \in M} \psi(\tau v')v \tau v'v} \vspace{2pt} \\
\phantom{y_{\psi}} = {\displaystyle \frac{n_q}{4} \sum_{v \in M} v \tau v}.
\end{array}
\end{equation}
In the last step we used that $\psi(\tau v')=0$ for $v' \neq 1$ and $\psi(\tau)=n_q$. For this, note that the trace of the matrices involved is $\lambda$, which is different from $\pm 2$. We now compute $\psi(y_{\psi}^2)$:
\begin{equation}\label{computation2}
\begin{array}{ll}
\psi(y_{\psi}^2) & = {\displaystyle \frac{n_q^2}{16} \sum_{v,v' \in M} \psi(v \tau vv' \tau v')} \vspace{2pt} \\
 & = {\displaystyle \frac{n_q^2}{16} \sum_{v,v' \in M} \psi(\tau vv' \tau vv')} \vspace{2pt} \\
 & = {\displaystyle \frac{n_q^2}{4} \sum_{v \in M} \psi\big((\tau v)^2\big)} \vspace{2pt} \\
 & = {\displaystyle \frac{n_q^3}{4}}.
\end{array}
\end{equation}
We\vspace{1pt} used in the final equality that $\psi((\tau v)^2)=0$ for $v \neq 1$ and $\psi(\tau^)=n_q$. Again note that the trace of the matrices involved is different from $\pm 2$ because $\lambda \notin \F_3$. \par \vspace{2pt}

In light of Propositions \ref{character} and \ref{decomp}, we have that $\frac{n_q^3}{4} \in R$. Since $\gcd(n_q,4)=1$, we obtain that $\frac{1}{4} \in R$. This establishes the statement for $p=3$. \par \vspace{4pt}

To deal with the rest of values of $p$, observe that $\PSL_2(p)$ is a subgroup of $\PSL_2(q)$. Recall that we took the pair $(x,y)$ in $\F_p \times \F_p$. Since $M$ is contained in $\PSL_2(p)$, by Proposition \ref{subsquo}(i), we can handle $\PSL_2(p)$ instead of $\PSL_2(q)$ for our purpose. \par \vspace{2pt}

2. Case $p=5$. This case was discussed in \cite[Theorem 3.3]{CM2} as $\PSL(2,5) \simeq A_5$. Nevertheless, we provide a proof in this context for completeness. We consider the pair $(x,y)=(2,0)$. Take $\tau=\pi \hspace{-1.5pt} \left(\begin{smallmatrix}
1 & 0 \\
1 & 1
\end{smallmatrix}\right)$. The condition $M \cap (\tau M \tau^{-1})=\{1\}$ holds. Pick a primitive $4$th root of unity $\eta$ and a primitive $6$th root of unity $\nu$ in $\Co$. Let $\alpha$ be the complex character of $C_4$ determined by $\eta^2$. Let $\theta$ be the complex character of $C_6$ determined by $\nu^2$. We work with the character $\phi=\textrm{St}+R(\alpha)+R'(\theta)$. Note that $\phi(1)=\phi(-1)$, so we can view $\phi$ as a character of $\PSL_2(5)$. Its values are given in the following table:\vspace{1.5mm}

\begin{center}
\scalebox{0.93}{$\begin{array}{|c|ccccc|} \hline
\textrm{Conjugacy class}  & \pi(1) & \pi(u_+)  & \pi(u_{-})   &   \pi(d(a))  & \pi(d'(\xi))  \\ \hline
\textrm{Order}            &  1     &  5        &     5        &        2        &       3       \\ \hline
\phi                      &  15    &  0        &     0        &        -1       &       0       \\ \hline
\end{array}$}\vspace{1.5mm}
\end{center}
Here, $a$ is an element in $\F_5^{\times}$ of order $4$ and $\xi$ is an element in $\F_{25}^{\times}$ of order $6$. One can see that $\tau\hspace{-1pt} s$ has order $2$, $\tau r$ has order $3$, and $\tau$ and $\tau rs$ have order $5$. Then, $\phi(\tau)=\phi(\tau r)=\phi(\tau rs)=0$ and $\phi(\tau\hspace{-1pt} s)=-1$. We calculate $y_{\phi}$ as in \eqref{computation1}:
$$y_{\phi} = \frac{1}{\vert M \vert} \sum_{v,v' \in M} \phi(\tau v')v \tau v'v = -\frac{1}{4} \sum_{v \in M} v (\tau\hspace{-1pt} s) v.$$
We now compute $\phi(y_{\phi}^2)$ as in \eqref{computation2}. We get:
$$\begin{array}{ll}
\phi(y_{\phi}^2) & \hspace{-2.5mm} = {\displaystyle \frac{1}{16} \sum_{v,v' \in M} \phi\big(v(\tau\hspace{-1pt} s v')^2 v\big)} \vspace{3pt} \\
 & \hspace{-2.5mm}= {\displaystyle \frac{1}{16} \sum_{v,v' \in M} \phi\big((\tau\hspace{-1pt} s v')^2 v^2\big)} \vspace{3pt} \\
 & \hspace{-2.5mm}= {\displaystyle \frac{1}{4} \sum_{v \in M} \phi\big((\tau\hspace{-1pt} s v)^2\big)} \vspace{3pt} \\
 & \hspace{-2.5mm}= {\displaystyle \frac{15}{4}}.
\end{array}$$
By Propositions \ref{character} and \ref{decomp}, we have that $\frac{1}{4} \in R$. This gives the statement of Theorem \ref{main2} for $p=5$. \par \vspace{4pt}
\enlargethispage{\baselineskip}

3. Case $p=7$. The proof of this case follows the lines of the preceding one. We work with the pair $(x,y)=(2,3)$. As before, we take $\tau=\pi \hspace{-1.5pt} \left(\begin{smallmatrix}
1 & 0 \\
1 & 1
\end{smallmatrix}\right)$, which satisfies $M \cap (\tau M \tau^{-1})=\{1\}$. Pick a primitive $6$th root of unity $\eta$ and a primitive $8$th root of unity $\nu$ in $\Co$. Let $\alpha$ be the complex character of $C_6$ determined by $\eta^2$. Let $\theta$ be the complex character of $C_8$ determined by $\nu^2$. The character of $\SL_2(7)$ that we will use is $\phi=R(\alpha)+R'(\theta)$, which can be viewed as a character of $\PSL_2(7)$ because $\phi(1)=\phi(-1)$. The values of $\phi$ are given in the following table: \vspace{1.5mm}

\begin{center}
\scalebox{0.93}{$\begin{array}{|c|cccccc|} \hline
\textrm{Conjugacy class}  & \pi(1)  & \pi(u_+)  & \pi(u_{-})   &   \pi(d(a))  & \pi(d'(\xi)) &  \pi(d'(\xi^2)) \\ \hline
\textrm{Order}            &  1      &  7        &     7        &        3     &       4      &      2               \\ \hline
\phi                      &  14     &  0        &     0        &        -1    &       0      &      2 \\ \hline
\end{array}$}\vspace{1.5mm}
\end{center}
Here, $a$ is an element in $\F_7^{\times}$ of order $6$ and $\xi$ is an element in $\F_{49}^{\times}$ of order $8$. One can check that $\tau r$ has order $3$, $\tau s$ has order $4$, and $\tau$ and $\tau rs$ have order $7$. Then, $\phi(\tau)=\phi(\tau\hspace{-1pt} s)=\phi(\tau rs)=0$ and $\phi(\tau r)=-1$. We compute $y_{\phi}$ as in \eqref{computation1}:
$$y_{\phi} = \frac{1}{\vert M \vert} \sum_{v,v' \in M} \phi(\tau v')v \tau v'v = - \frac{1}{4} \sum_{v \in M} v (\tau r) v.$$
We now compute $\phi(y_{\phi}^2)$ as in \eqref{computation2}. We get:
$$\phi(y_{\phi}^2) = \frac{1}{16} \sum_{v,v' \in M} \phi\big((\tau r v')^2 v^2\big) = \frac{1}{4} \sum_{v \in M} \phi\big((\tau r v)^2\big) = \frac{1}{4}.$$
By Propositions \ref{character} and \ref{decomp}, we have that $\frac{1}{4} \in R$. This establishes the statement of Theorem \ref{main2} for $p=7$. \par \vspace{4pt}

4. Case $p>7$. We need the following lemma, which will be useful in detecting that certain elements do not belong to $P$:

\begin{lemma}\label{lambda}
Let $(x,y) \in \F_p \times \F_p$ be such that $x^2+y^2=-1$. Then, there is $\lambda \in \F_p^{\times}$ such that $\lambda, \lambda x,$ and $\lambda y$ are all different from $\pm 2$.
\end{lemma}

\begin{proof}
Note that either $x \neq \pm 2$ or $y \neq \pm 2$. Otherwise, we would have $2^2+2^2=-1$, which would give $p=3$.
\par \vspace{2pt}

If $x \neq \pm 2$ and $y \neq \pm 2$, then we can take $\lambda=1$ and we are done. Suppose that $x=\pm 2$ and $y \neq \pm 2$. Then,
$-1 = x^2+y^2 =4+y^2$. This\vspace{1pt} implies $y^2 = -5$, and thus $(\lambda y)^2 = -5 \lambda^2.$ From the latter, we deduce that $\lambda y \neq \pm 2$ if and only\vspace{1pt} if $\lambda^2 \neq -5^{-1} \cdot 4$. We also want $\lambda^2 \neq 4$. So, we must choose $\lambda$ such that $\lambda^2 \notin  \{1,4,-5^{-1} \cdot 4\}$. The number\vspace{1pt} of squares in $\F_p^{\times}$ is $\frac{p-1}{2}$, and $\frac{p-1}{2}>3$ if $p>7$. This guarantees a choice of $\lambda$ with the required properties. The argument is the same for $x \neq \pm 2$ and $y=\pm 2$.
\end{proof}

We pick $\lambda$ satisfying the conclusion of Lemma \ref{lambda}. Take $\tau = \pi \hspace{-1.5pt} \left(\begin{smallmatrix}
1 & 0 \\ \lambda & 1 \end{smallmatrix}\right)$. One can verify that our choice of $\lambda$ ensures that $M \cap (\tau M \tau^{-1})=\{1\}$. On the other hand, the cardinality of the set $\{(x,y) \in \F_p \times \F_p : x^2+y^2=-1\big\}$ is $p+1$ if $p \equiv_4 -1$ and $p-1$ if $p \equiv_4 1$. In view of Proposition \ref{subPSL2}(ii), we can assume, for our purpose, that the pair $(x,y)$ chosen to define the element $s$ in $M$ satisfies $xy \neq 0$. \par \vspace{2pt}

The same computations of \eqref{computation1} and \eqref{computation2} give:
$$y_{\psi} = \frac{n_p}{4} \sum_{v \in M} v \tau v \quad \textrm{ and } \quad \psi(y_{\psi}^2) = \frac{n_p^3}{4}.$$
For the first equality we used that $\psi(\tau)=n_p$ and $\psi(\tau v)=0$ for $v \neq 1$. The latter holds because the trace of the matrices involved is $\lambda, \lambda x,$ and $\lambda y$, which is different from $\pm 2$ by assumption. In the second equality, we used:
\begin{enumerate}
\item That $\tau^2, (\tau r)^2, (\tau s)^2,$ and $(\tau rs)^2$ are different from $1$. Note that $(\tau\hspace{-0.5pt} s)^2=1$ implies $y=0$. Similarly, $(\tau r s)^2=1$ implies $x=0$. This would contradict our choice of $(x,y)$. \vspace{2pt}

\item That $\psi$ vanishes on $(\tau r)^2, (\tau s)^2,$ and $(\tau rs)^2$. For this, we argue\vspace{0.5pt} as follows. If $\psi$ did not vanish on $(\tau r)^2$, then $(\tau r)^2$ would have order $p$. This would give that $\tau r$ would have order $p$. But the trace of the matrix
     $$\hspace{12mm}\left(\begin{smallmatrix}
1 & 0 \\ \lambda & 1
\end{smallmatrix}\right)\hspace{-2pt} \left(\begin{smallmatrix}
0 & 1 \\ -1 & 0
\end{smallmatrix}\right) =  \left(\begin{smallmatrix}
0 & 1 \\ -1 & \lambda
\end{smallmatrix}\right)$$ is different from $\pm 2$ by our choice of $\lambda$. The same applies to $(\tau s)^2$ and $(\tau rs)^2$ as:
    $$\hspace{12mm}\left(\begin{smallmatrix}
1 & 0 \\ \lambda & 1
\end{smallmatrix}\right)\hspace{-2pt} \left(\begin{smallmatrix}
x & \hspace{2pt} y \\ y & -x
\end{smallmatrix}\right) =  \left(\begin{smallmatrix}
x & \hspace{2pt} y \\ \lambda x+y & \hspace{2pt} \lambda y -x
\end{smallmatrix}\right) \quad \textrm{ and } \quad
\left(\begin{smallmatrix}
1 & 0 \\ \lambda & 1
\end{smallmatrix}\right)\hspace{-2pt} \left(\begin{smallmatrix}
\hspace{-1pt}-y & x \\ \hspace{1pt} x & y
\end{smallmatrix}\right) =  \left(\begin{smallmatrix}
-y & \hspace{2pt} x \\ -\lambda y+x & \hspace{2pt} \lambda x+y
\end{smallmatrix}\right).$$
\end{enumerate}

By Propositions \ref{character} and \ref{decomp}, we have that $\psi(y_{\psi}^2) \in R$. Since $\gcd(n_p,4)=1$, we obtain that $\frac{1}{4} \in R$. This finishes the proof of Theorem \ref{main2}. \qed

\begin{remark}
When $p \equiv_4 -1$ any solution of $x^2+y^2=-1$ satisfies $xy \neq 0$. When $p \equiv_4 1$, there are solutions such that $xy=0$. We can deal directly with this situation by modifying our argument in the following way. At most one of $(\tau\hspace{-1pt} s)^2$ and $(\tau rs)^2$ is $1$. Then:
$$\psi(y_{\psi}^2) = \Big(\frac{p+1}{4}\Big)^2 \sum_{v \in M} \psi\big((\tau v)^2\big) = \Big(\frac{p+1}{4}\Big)^2 \Big(\gamma+\frac{p+1}{2}\Big),$$
where $\gamma \in \big\{0, \frac{1-p^2}{2}\big\}$. We obtain:
$$\psi(y_{\psi}^2) = \left\{
\begin{array}{ll}
{\displaystyle \Big(\frac{p+1}{4}\Big)^3}  & \textrm{ if } \hspace{5pt} \gamma=0, \vspace{5pt} \\
{\displaystyle \frac{(p+1)^3(2-p)}{32}} & \textrm{ if } \hspace{5pt} {\displaystyle \gamma= \frac{1-p^2}{2}}.
\end{array}\right.$$
Since $p \equiv_4 1$, it follows that $\frac{1}{2} \in R$.
\end{remark}

As was the case with Corollary \ref{complexif}, the following result is coupled with Theorem \ref{main2}:

\begin{corollary}\label{cormain2}
Let $\Omega$ be a non-trivial twist of $\Co \PSL_2(q)$ arising from a $2$-cocycle on an abelian subgroup of $\PSL_2(q)$. Then, the complex semisimple Hopf algebra $(\Co \PSL_2(q))_{\Omega}$ does not admit a Hopf order over any number ring.
\end{corollary}

This provides a second family of simple groups for which the non-existence of integral Hopf orders is established for \emph{any} non-trivial twist arising from an abelian subgroup. \par \vspace{2pt}

This corollary serves as an introduction for the discussion of the last section.

\section{Twists of finite non-abelian simple groups}
\label{applications}

Let $G$ be a finite non-abelian simple group and $\Omega$ a non-trivial twist for $\Co G$ arising from a $2$-cocycle on an abelian subgroup of $G$. In \cite[Question 5.1]{CM2} it was asked whether $(\Co G)_{\Omega}$ can admit a Hopf order over a number ring. We can partially answer this question in the negative building on Theorems \ref{main1} and \ref{main2} and two results on the subgroups structure of finite non-abelian simple groups. \par \vspace{2pt}

Recall from \cite[Section 2]{Th} that a \emph{minimal simple group} is a non-abelian simple group all of whose proper subgroups are solvable. The following remarkable classification was established in \cite[Corollary 1, page 388]{Th}:

\begin{theorem}[Thompson]\label{Thompson}
Every minimal simple group is isomorphic to one of the following groups:
\begin{enumerate}
\item[(i)]  $\PSL_2(2^p)$, with $p$ a prime. \vspace{1pt}
\item[(ii)] $\PSL_2(3^p)$, with $p$ an odd prime. \vspace{1pt}
\item[(iii)] $\PSL_2(p)$, with $p>3$ prime such that $5$ divides $p^2+1$. \vspace{1pt}
\item[(iv)] ${}^2\!B_2(2^p)$, with $p$ an odd prime. \vspace{1pt}
\item[(v)] $\PSL_3(3)$.
\end{enumerate}
\end{theorem}

Relying on this and the classification of the finite simple groups, the following result was proved in \cite[Theorem 1]{BW}:

\begin{theorem}[Barry-Ward]
Every finite non-abelian simple group contains a minimal simple group as a subgroup.
\end{theorem}
\enlargethispage{\baselineskip}

Theorems \ref{main1} and \ref{main2} and Proposition \ref{subsquo}(i), reinforced with the previous two results, give as a consequence:

\begin{theorem}\label{simple1}
Let $K$ be a number field and $G$ a finite non-abelian simple group. Then, there is a twist $\Omega$ for $K\hspace{-1pt}G$, arising from a $2$-cocycle on an abelian subgroup of $G$, such that $(K\hspace{-1pt}G)_{\Omega}$ does not admit a Hopf order over $\Oint_K$.
\end{theorem}

The statement for the complexified group algebra now follows as in Corollary \ref{complexif}:

\begin{corollary}\label{simple2}
Let $G$ be a finite non-abelian simple group. Then, there is a twist $\Omega$ for $\Co G$, arising from a $2$-cocycle on an abelian subgroup of $G$, such that $(\Co G)_{\Omega}$ does not admit a Hopf order over any number ring.
\end{corollary}

For the sporadic groups or the Tits group, Theorem \ref{simple1} and Corollary \ref{simple2} can be deduced from \cite[Theorem 3.3, Remark 3.4, and Corollary 3.5]{CM2} in view of the following remark:

\begin{remark}
Let $G$ be a sporadic group or the Tits group $^2 F_4(2)'$. Then, $G$ has a subgroup isomorphic to $A_5$.
\end{remark}

This remark can be verified by inspection of the tables of maximal subgroups for the sporadic groups in \cite[Section 4]{W2} and that for the Tits group in \cite[Theorem 1]{W1} and \cite{Tc}. A close look reveals the inclusions as listed below:\vspace{1.5mm}
\begin{center}
\scalebox{0.86}{\begin{tabular}{|c|c|p{1mm}|c|c|p{1mm}|c|c|}
\cline{1-2} \cline{4-5} \cline{7-8}
\textrm{Group} & \textrm{Contains } &  & \textrm{Group} & \textrm{Contains } &  & \textrm{Group} & \textrm{Contains } \\
\cline{1-2} \cline{4-5} \cline{7-8}
$M_{11}$ & $S_5$ & & $M_{12}$ & $S_5$    & & \multicolumn{1}{c|}{\multirow{2}{*}{$M_{22}$}} & \multicolumn{1}{c|}{\multirow{2}{*}{$A_7$}} \\
$M_{23}$ & $A_8$ & & $M_{24}$ & $A_8$ & &  \multicolumn{1}{c|}{ }        &    \multicolumn{1}{c|}{ }    \\
\cline{1-2} \cline{4-5} \cline{7-8}
$J_1$    & $A_5$    & & \multicolumn{1}{c|}{\multirow{2}{*}{$J_2$}}    & \multicolumn{1}{c|}{\multirow{2}{*}{$A_5$}}     & &\multicolumn{1}{c|}{\multirow{2}{*}{$J_3$}}     &  \multicolumn{1}{c|}{\multirow{2}{*}{$A_5$}}  \\
$J_4$    & $M_{22}$ & & \multicolumn{1}{c|}{ }                        &  \multicolumn{1}{c|}{ }         &   & \multicolumn{1}{c|}{ } & \multicolumn{1}{c|}{ } \\
\cline{1-2} \cline{4-5} \cline{7-8}
$C\hspace{-0.75pt}o_1$   & $A_9$    & & $C\hspace{-0.75pt}o_2$   & $M_{23}$ & & $C\hspace{-0.75pt}o_3$   & $S_5$   \\
\cline{1-2} \cline{4-5} \cline{7-8}
$F\hspace{-0.75pt}i_{22}$ & $S_{10}$ & & $F\hspace{-0.75pt}i_{23}$ & $S_{12}$ & & $Fi_{24}'$ & $A_5$\\
\cline{1-2} \cline{4-5} \cline{7-8}
$H\hspace{-0.75pt}S$ & $S_8$ & & $M\hspace{-0.75pt}cL$ & $A_7$ & & $H\hspace{-0.75pt}e$   & $S_4(4)$ \\
$Ru$ & $A_6$ & & $Suz$ & $A_7$ & & $O'\hspace{-0.75pt}N$ & $A_7$ \\
$H\hspace{-0.75pt}N$ & $A_{12}$ & & $Ly$ & $M_{11}$ & & $T\hspace{-0.75pt}h$ &  $S_5$ \\
$B$  & $S_5$ & & $\mathbb{M}$ & $A_5$ & & $^2\!F_4(2)'$ & $\PSL_2(5^2)$ \\
\cline{1-2} \cline{4-5} \cline{7-8}
\end{tabular}}\vspace{1.5mm}
\end{center}
For the Tits group, note that $\PSL_2(5^2)$ contains $\PSL_2(5)$, which is isomorphic to $A_5$; see \cite[Theorem 6.26(iv), page 414]{Su2} for a more general statement.

\begin{remark}
Theorem \ref{simple1} and Corollary \ref{simple2} hold true for any finite group containing a non-abelian simple group in light of Proposition \ref{subsquo}(i). Among the groups satisfying this condition we find almost simple groups and some families of primitive groups, see \cite[Section 4.8]{DM} for more details.
\end{remark}

In our way to Corollary \ref{simple2} we showed that the complex group algebra of any finite non-abelian simple group can be twisted to produce a simple non-commutative and non-cocommutative Hopf algebra. This was first proved by Hoffman in \cite{H} following a different strategy that does not use minimal simple groups.

\end{document}